\documentclass[12pt]{article}

\usepackage{amsthm,amsmath,amssymb}
\usepackage[english]{babel}
\usepackage[utf8]{inputenc}
\usepackage{mathtools}
\usepackage{authblk}
\usepackage{xcolor} \definecolor{bluegreen2}{RGB}{0, 85, 127}
\usepackage[linktocpage]{hyperref}
\hypersetup{colorlinks=true, linkcolor=bluegreen2, citecolor=bluegreen2, filecolor=magenta, urlcolor=bluegreen2}
\usepackage{tikz} \usetikzlibrary{spy,arrows,trees,shapes,calc,cd}
\usepackage{pgfplots}
\usepackage[footnotesize]{caption}
\usepackage{enumitem}
\pgfplotsset{every axis/.append style={
                    axis x line=middle,    
                    axis y line=middle,    
                    axis line style={-,color=blue}, 
                    xlabel={$x$},          
                    ylabel={$y$},          
            }}
\pgfplotsset{compat=1.13}
\usepackage[title]{appendix}
\usepackage{cancel}
\usepackage{ifoddpage}

\theoremstyle{plain}
\newtheorem{thm0}{Theorem}
\newtheorem{thm}{Theorem}[section]
\newtheorem{prop}[thm]{Proposition}
\newtheorem{lemma}[thm]{Lemma}
\newtheorem{cor}[thm]{Corollary}
\theoremstyle{definition}
\newtheorem{dfn}[thm]{Definition}
\newtheorem{ex}[thm]{Example}
\theoremstyle{remark}
\newtheorem{remark}[thm]{Remark}

\DeclareMathOperator{\Sing}{Sing}
\DeclareMathOperator{\Res}{Res}
\DeclareMathOperator{\Eh}{L}
\DeclareMathOperator{\mult}{mult}
\DeclareMathOperator{\ord}{ord}
\DeclareMathOperator{\lct}{lct}
\newcommand{\floor}[1]{\left \lfloor #1 \right \rfloor}
\newcommand{\fpart}[1]{\left \{ #1 \right \}}
\newcommand{\PP}{\mathbb{P}}
\newcommand{\CC}{\mathbb{C}}
\newcommand{\RR}{\mathbb{R}}
\newcommand{\QQ}{\mathbb{Q}}
\newcommand{\ZZ}{\mathbb{Z}}
\newcommand{\NN}{\mathbb{N}}
\newcommand{\cO}{\mathcal{O}}
\newcommand{\cK}{\mathcal{K}}
\newcommand{\cC}{\mathcal{C}}
\def\Diff{\Delta}
\def\cD{T}
\def\w{w}

\title{\bf Counting points with Riemann-Roch formulas}
\author{Jorge Mart\'{\i}n-Morales\thanks{
The author is partially supported
by the European Union NextGenerationEU/PRTR (grant code: RYC2021-034300-I),
by MCIN/AEI/10.13039/501100011033 (grant code: PID2020-114750GB-C31),
by Departamento de Ciencia, Universidad y Sociedad del Conocimiento del Gobierno de Aragón (grant code: E22 20R: ``Algebra y Geometría''),
and by Junta de Andalucía (grant code: FQM-333). \medskip
}}
\affil{\small Department of Mathematics, IUMA, University of Zaragoza \\
Calle Pedro Cerbuna 12, 50009, Zaragoza, Spain \\
URL: \url{http://riemann.unizar.es/\~jorge} \\
email: \href{mailto:jorge.martin@unizar.es}{jorge.martin@unizar.es}}
\date{}

\begin{document}

\maketitle

\begin{abstract}
We  provide an algorithm for computing the number of integral points lying in certain triangles
that do not have integral vertices. We use techniques from Algebraic Geometry such as
the Riemann-Roch formula for weighted projective planes and resolution of singularities.
We analyze the complexity of the method and show that the worst case is given by
the Fibonacci sequence. At the end of the manuscript a concrete example is developed in detail
where the interplay with other invariants of singularity theory is also treated.
\end{abstract}

\renewcommand{\thefootnote}{}
\footnote{ \medskip }   



\section*{Introduction}

This paper deals with the general problem of counting lattice points in a polyhedron
with rational vertices and its connection with both singularity theory of
surfaces and adjunction formulas and Riemann-Roch formulas for curves in the weighted projective plane.
In addition, we focus on rational polyhedra whose vertices are rational points as opposed
to lattice polyhedra whose vertices are integers. Our approach exploits the
connection between Dedekind sums \cite{RG72} and geometry of cyclic quotient singularities,
which has been proposed by several authors \cite{Pom93, Bla95, Bri95, Lat95, Ash15}.
Other important references about the subject are
\cite{Bar07}, \cite{Bar97}, \cite{Ewa96}, \cite{Oda88}, \cite{DHTY04}, \cite{DeL05}.

A polyhedron is a three-dimensional shape in $\RR^3$ with flat polygonal faces, straight edges, and sharp vertices.
Common examples are cubes, prisms, pyramids, etc. For instance, a classical result shows that only five convex regular
polyhedra exist, namely the five Platonic solids, see Figure~\ref{fig:five-platonic-solids}.
However, cones and spheres are not polyhedra since they do not have polygonal faces.
The generalization of polyhedra to higher dimension in~$\RR^n$ are called polytopes.
This way a polygon is simply a two-dimensional polytope.

\begin{figure}[ht]
\centering
\includegraphics[scale=0.41]{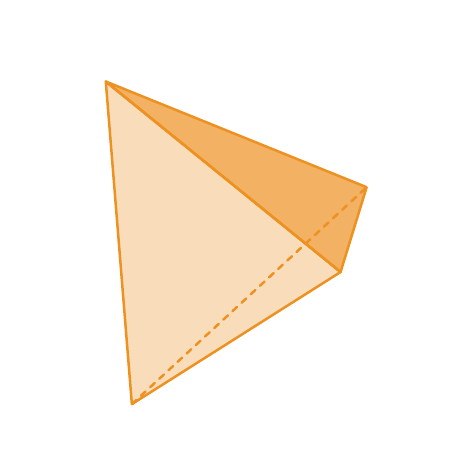} \hspace{-22pt}
\includegraphics[scale=0.41]{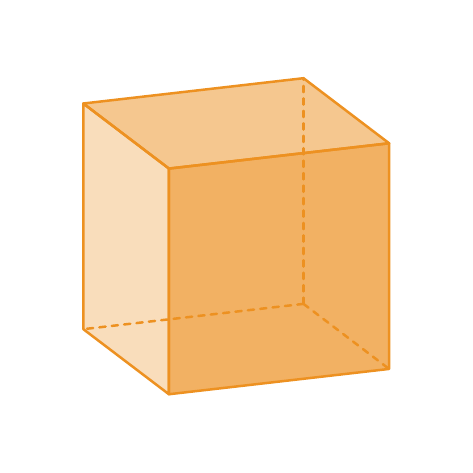} \hspace{-12pt}
\includegraphics[scale=0.63]{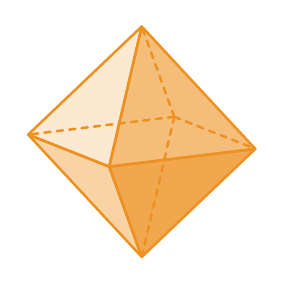}
\includegraphics[scale=0.37]{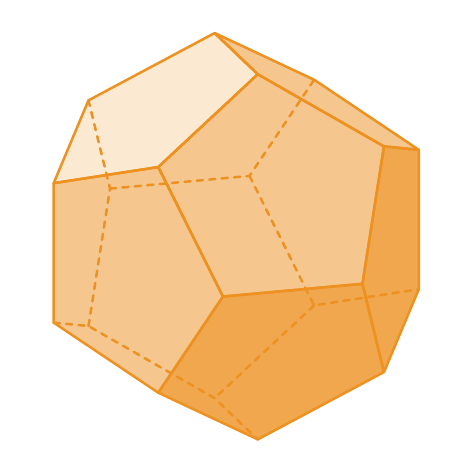} \hspace{1pt}
\includegraphics[scale=0.36]{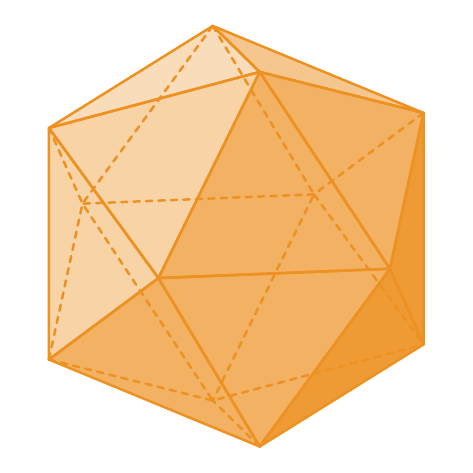}
\caption{The five Platonic solids.}
\label{fig:five-platonic-solids}
\end{figure}

Combinatorics is a branch of mathematics which is concerned with the study
of finite or countable discrete sets. One of the interests in combinatorics is the
counting of certain elements in a given set. The main question addressed in this work
is how to exact count the number of points with integral coordinates inside a convex bounded polytope.
Perhaps the most famous case is the theory of Ehrhart polynomials, introduced
by Eugène Ehrhart \cite{Ehr77}, see also~\cite{CLS12}. These polynomials count the number of lattice
points in the different integral dilations of an integral convex polytope.
We emphasize that whenever we say counting, we mean exact counting.
There is a rich and exciting theory of estimation and approximation,
but that is a very different subject.

A wide variety of topics in Mathematics involve this challenging and hard problem. 
Counting integral points in polyhedra or other questions about them arise in Representation Theory,
Commutative Algebra, Algebraic Geometry, Statistics, and Computer Science.
Applications range from the very pure such as number theory, Hilbert functions, and Kostant's partition
function in representation theory, to the most applied such as cryptography, integer programming, and
contingency tables. Another interesting application is to voting theory which
is concerned with elections and voting systems~\cite{Sch13}.
If we try to count lattice points in more complicated regions of $\RR^4$, then we can find
applications to RSA cryptography~\cite{Sal23}.

The simplest example has successfully been studied by Pick in 1899 \cite{Pick1899}. Pick's Theorem provides
a formula for the area of a simple polygon with integral vertices in terms of the number
within it and on its boundary. There are multiple proofs and they can be generalized to formulas
for some non-simple polygons. More precisely, suppose that a polygon has integral
coordinates for all of its vertices. Let $i$ be the number of integral points interior to the polygon
and let $b$ be the number of integral points on its boundary. Then the area $A$ of this polygon is
\[
A = i + \frac{b}{2} - 1.
\]
Figure~\ref{fig:pick-thm} shows an example of a triangle in $\RR^2$ with vertices $(1,1)$, $(5,2)$, $(3,5)$
where the number of interior points is $i = 6$, the number of points on the boundary is $b = 4$,
and therefore its area is $A = 7$.

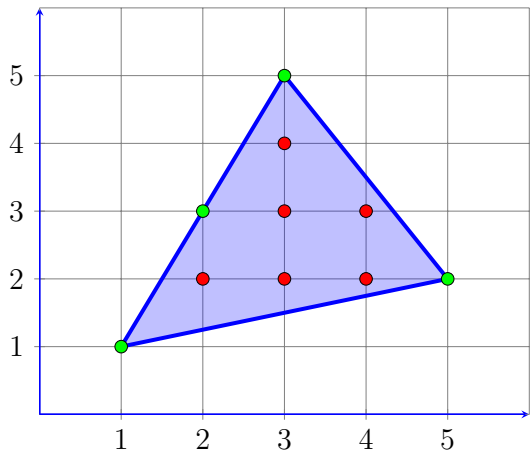
\begin{figure}[ht]
\centering
\begin{tikzpicture}[xscale=0.95, yscale=0.95]
\begin{axis}[domain=0:6, xmin=0, xmax=6, ymin=0, ymax=6, xlabel={}, ylabel={},
axis x line=center, axis y line=center, xtick={1,2,3,4,5}, ytick={1,2,3,4,5}, samples=100]
\draw [help lines,step=1] (0,0) grid (6,6);
\draw [draw=blue, thick] (0,0) -- (5.95,0);
\draw [draw=blue, thick] (0,0) -- (0,5.95);
\draw [white, fill=blue, opacity=0.25] (1,1) -- (5,2) -- (3,5) -- cycle;
\draw [blue, ultra thick] (1,1) -- (5,2) -- (3,5) -- cycle;
\draw [draw=black, fill=green] (1,1) circle (2.5pt);
\draw [draw=black, fill=green] (5,2) circle (2.5pt);
\draw [draw=black, fill=green] (3,5) circle (2.5pt);
\draw [draw=black, fill=green] (2,3) circle (2.5pt);
\draw [draw=black, fill=red] (2,2) circle (2.5pt);
\draw [draw=black, fill=red] (3,2) circle (2.5pt);
\draw [draw=black, fill=red] (4,2) circle (2.5pt);
\draw [draw=black, fill=red] (3,3) circle (2.5pt);
\draw [draw=black, fill=red] (4,3) circle (2.5pt);
\draw [draw=black, fill=red] (3,4) circle (2.5pt);
\end{axis}
\end{tikzpicture}
\caption{Pick's theorem for a triangle.}
\label{fig:pick-thm}
\end{figure}

For a precise description of the main results we present in this work,
some notation needs to be introduced.
In this paper we focus on the study of the point counting problem
and the complexity of an algorithm to find the number of integral points
for polygons of type
\[
\cD_{\w,d} = \{ (x,y,z) \in \RR_{\geq 0}^3 \mid w_0 x + w_1 y + w_2 z = d \},
\]
where $w_0, w_1, w_2, d \in \ZZ_{\geq 0}$.
Assume for a moment $w_0,w_1,w_2$ are pairwise coprime integers and
denote by $w=(w_0,w_1,w_2)$, $|w|=w_0+w_1+w_2$, and $\bar{w}=w_0w_1w_2$. 
The case where the weights are not necessarily pairwise coprime is treated
in \S\ref{sec:non-coprime}.

Consider $\PP^2_w$ the weighted projective plane. For a given Weil divisor $D$ in $\PP^2_w$,
$\cO_{\PP^2_w}(D)$ denotes the sheaf $\cO_{\PP^2_w}(D) = \{ f \in \cK_{\PP^2_w} \mid (f) + D \geq 0 \}$,
being $\cK_{\PP^2_w}$ the sheaf of rational function on $\PP^2_w$.
Finally $K_{\PP^2_w}$ denotes the canonical divisor of $\PP^2_w$.
One of the key ingredients to connect the arithmetical problem referred above with the 
geometry of weighted projective planes comes from the observation that 
\begin{equation*}
\Eh_{\w}(d) := \# (\cD_{\w,d}\cap \ZZ^3) = \chi(\PP^2_w, \cO_{\PP^2_w}(D)) = h^0(\PP^2_\w, \cO_{\PP^2_w}(D)),
\end{equation*}
where $D$ is a Weil divisor in $\PP^2_w$ of degree $d$. In other words $\Eh_{\w}(d)$ coincides
with the dimension of the vector space of weighted homogeneous polynomials of degree~$d$.
This way points $(i,j,k)$ from the lattice $\cD_{\w,d}\cap \ZZ^3$ correspond to monomials
$x^i y^j z^k$ in $\CC[x,y,z]$ of weighted degree $d$.

Let us present the main results of this work. The first main statement revisits the Riemann-Roch
formula and presents a new simplified proof for weighted projective planes.
In particular, it shows an explicit formula for the Ehrhart quasi-polynomial $\Eh_\w(d)$
of degree two of $\cD_{\w,d}$ in terms of~$d$.

\begin{thm0}\label{thm:RR-formula}
Let $D$ be a divisor in $\PP^2_w$ of degree $d$. Then
\[
\chi(\PP^2_w, \cO_{\PP^2_w}(D)) = 1 + \frac{1}{2} D \cdot (D - K_{\PP^2_w}) - \sum_{P\in\Sing(\PP^2_w)}\Diff_P(d+|\w|),
\]
where $\Sing(\PP^2_w)$ denotes the singular locus of the weighted projective plane.
\end{thm0}

The quadratic term $\frac{1}{2} D \cdot (D-K_{\PP^2_w}) = \frac{1}{2}d(d+|w|)$ has to do
with the virtual genus of a curve and $\Diff_P(k)$
is a periodic function of period $\bar \w$ which is an invariant associated with the singularity 
$P\in \Sing(\PP^2_w)$, see~\cite{CMO14, CMO16, CM19}.

The previous combinatorial number $\Diff_P(k)$ has a geometric interpretation and it can be computed via
invariants of curve singularities on a singular surface as follows. Let $(f,P)$ be a reduced curve germ
at a point $P$ in a surface $X$ with a cyclic quotient singularity. Then 
\begin{equation}\label{eq:Delta-invariant}
\Delta_X (k) = \delta_X^{\text{top}}(f) - \kappa_X(f)
\end{equation}
for any reduced germ $f\in \cO_{X,P}(k)$.
Here $\delta_X^{\text{top}}$ is the topological delta invariant
and $\kappa_X$ is the analytic kappa invariant of the singularity.
Note that the choice of a reduced $f \in \cO_X(k)$ does not affect the
result of $\Delta_X(k)$. In Blache's the notation $A_X(k) = \Diff_X(k)$ and $R_X(D) = -\Delta_X(d+|w|)$,
\cite[\S2.1]{Bla95}.
This way the Riemann-Roch formula of Theorem~\ref{thm:RR-formula} can be rewritten as
\begin{equation*}
\chi(\PP^2_w, \cO_{\PP^2_w}(D)) = 1 + \frac{1}{2} D \cdot (D - K_{\PP^2_w}) + R_{\PP^2_w}(D)
\end{equation*}
where $R_X(D)$ is called the correction term, cf.~\cite{Bre77, Bla95}.
As a by-product we obtain a new expression for the correction term,
see~\eqref{eq:R-root-unity},
\[
R_{(w_2;\, w_0,w_1)}(d) := 
- \frac{1}{w_2} \sum_{i=1}^{w_2-1} \frac{1-\zeta_{w_2}^{-id}}{(1-\zeta_{w_2}^{iw_0})(1-\zeta_{w_2}^{iw_1})}.
\]

As an immediate consequence of Theorems~\ref{thm:RR-formula} and~\eqref{eq:Delta-invariant}
one has a method to compute $\Eh_\w(d)$ by means of appropriate curve germs $(\{f=0\},P)$
on surface quotient singularities.
The next results aims to show that the correction term $R_X$ or equivalently the $\Delta_X$-invariant 
can be computed following the Euclidean division algorithm.
\begin{thm0}\label{thm:correction-term}
Consider the division $d = c \cdot q + r$ with $0 \leq r < q$. Then, for $0 \leq k < d$, one has
\[
R_{X(d;1,q)}(k) = -R_{X(q;1,r)}(k) - \fpart{\frac{k}{q}} - \frac{k(k+1+q-d)}{2dq},
\]
where $\fpart{\frac{k}{q}} \in [0,1)$ denotes the decimal part of the fraction.
\end{thm0}

Note that Theorem~\ref{thm:correction-term} provides an effective method to compute
the correction term by repeatedly applying the Euclidean division algorithm as if we were
computing the greatest common divisor of two integers. From a computation point of view
it is known that the worst case of the Euclidean algorithm <is given by the Fibonacci sequence.

This paper is organized as follows. In~\S\ref{sec:counting} we present the main problem and apply the residue
theorem to find a numerical Riemann-Roch formula in \eqref{eq:numerical-RR} and a new expression of the correction
term in terms of certain sums of roots of unity, see~\eqref{eq:R-root-unity} and~\eqref{eq:RR-3terms}.
In~\S\ref{sec:Euler-intersection} we give a geometric interpretation of some terms appearing the Riemann-Roch formula
using the Euler characteristic of a sheaf and intersection theory in the weighted projective plane.
In~\S\ref{sec:DELTA} we prove the main result of this work, namely Theorem \ref{thm:RR-formula}, after
introducing the local $\Delta_P$-invariant of a divisor.
The second main result, Theorem \ref{thm:correction-term}, is presented in~\S\ref{sec:effective-correction-term}
where we study the effective computation of the correction term.
The non-pairwise coprime case is addressed in \S\ref{sec:non-coprime}.
Finally, \S\ref{sec:overview} is devoted to overviewing all the theory with a concrete example
where the interplay with other invariants of singularity theory is also treated.

\vspace{5pt}

\noindent\textbf{Acknowledgments.}
I wrote this work on the occasion of the Distinguished Researcher Award from
the Royal Academy of Exact, Physical, Chemical, and Natural Sciences of Zaragoza.
To be considered for an Academy Award is a great honor and a privilege, so thank you Academy.
It means so much to me and I am truly grateful.
I have been working on this topic with M.~Avendaño and J.I.~Cogolludo.
It is always a pleasure collaborating with them with their fruitful discussions
and ideas. I also thank E.~León-Cardenal for his final proofreading. And last but not least
I am deeply grateful to my family for their unconditional support.


\section{Counting points with the residue theorem}\label{sec:counting}

The main ideas behind this sections were inspired from~\cite{Bec00}, see also \cite{BDR02, BR07}.

Let $w_0, w_1, w_2, d \in \ZZ_{\geq 0}$. Assume $w_0,w_1,w_2$ are pairwise coprime integers and
denote by $w=(w_0,w_1,w_2)$, $|w|=w_0+w_1+w_2$, and $\bar{w}=w_0w_1w_2$.
The case where the weights are not necessarily pairwise coprime is treated in section~\ref{sec:non-coprime}.
Consider the triangle
\[
\cD_{\w,d} = \{ (x,y,z) \in \RR_{\geq 0}^3 \mid w_0 x + w_1 y + w_2 z = d \}.
\]
One is interested in computing the number of integers lying on the triangle, that is,
\begin{equation*}
\Eh_{\w}(d) := \# (\cD_{\w,d}\cap \ZZ^3).
\end{equation*}
A typical strategy to study the behavior of a sequence of cardinalities is to study its associated formal power series $\sum_{d \geq 0} \Eh_{\w}(d) t^d$. This can easily be done with a simple observation. Recall that
\[
\frac{1}{1-t} = 1 + t + t^2 + \cdots = \sum_{i \geq 0} t^i.
\]
Hence
\[
\frac{1}{(1-t^{w_0})(1-t^{w_1})(1-t^{w_2})} = \sum_{i,j,k \geq 0} t^{w_0 i + w_1 j + w_2 k} = \sum_{\ell \geq 0} \Eh_{\w}(\ell) t^\ell.
\]
Then
\[
f(t) := \frac{t^{-d-1}}{(1-t^{w_0})(1-t^{w_1})(1-t^{w_2})} = \Eh_{\w}(0) t^{-d-1} + \Eh_{\w}(1) t^{-d} + \cdots
+ \Eh_{\w}(d) t^{-1} + \cdots
\]
and therefore
\begin{equation}\label{eq:Lwd-integral}
\Eh_{\w}(d) = \Res(f(t),t=0) = \frac{1}{2\pi i} \oint_{\gamma} f(t) dt,
\end{equation}
where $\gamma$ is any circle of radius $\varepsilon$ (small enough) around $t=0$.

Now we can use residue techniques to study the number of points lying on the triangle.
First we modify a bit the function $f(t)$ as follows
\[
f(t) = g(t) + \frac{1}{(1-t^{w_0})(1-t^{w_1})(1-t^{w_2})t},
\qquad g(t) := \frac{t^{-d}-1}{(1-t^{w_0})(1-t^{w_1})(1-t^{w_2})t}.
\]
It is clear that
\begin{equation}\label{eq:Resft0}
\Res(f(t),t=0) = \Res(g(t),t=0) + 1.
\end{equation}
Let us fix $\zeta_{w_\ell}$ three primitive roots of unity of orders $w_\ell$, $\ell = 0,1,2$.
The poles of $g(t)$ are $t=0$ and $t=\zeta_{w_\ell}^i$, $i=0,1,\ldots,w_\ell-1$, $\ell=0,1,2$.
Summing up all the residues including the point at infinity one obtains
\begin{equation}\label{eq:sum-all-res}
\Res(g(t),t=0) + \sum_{i,\ell} \Res(g(t),t=\zeta_{w_\ell}^i) + \Res(g(t),t=\infty) = 0.
\end{equation}
Note that the residue of $g(t)$ at infinity is zero. Equations \eqref{eq:Lwd-integral}, \eqref{eq:Resft0},
\eqref{eq:sum-all-res} provides
\[
\Eh_{\w}(d) = 1 - \sum_{i,\ell} \Res(g(t),t=\zeta_{w_\ell}^i).
\]
The rest of this section is devoted to computing the residues appearing in this formula.
In the discussion we need to separate $t=1$ from the other roots since it has order~$2$ as a pole of~$g(t)$.
Note that the other poles are simple because the weights are pairwise coprime.

\subsection{Residue at \texorpdfstring{$t=1$}{t=1}}

Performing the change of variables $t = e^z$, one passes from a residue at $t=1$ to a residue at $t=0$
as follows
\[
\Res(g(t), t=1) = \Res(e^z g(e^z), t=0)
= \Res \left( \frac{e^{-dz}-1}{(1-e^{w_0z})(1-e^{w_1z})(1-e^{w_2z})}, t=0 \right).
\]
In order to compute this residue, we can use these two series
\[
\begin{aligned}
\frac{1}{1-e^z} &= -\frac{1}{z} + \frac{1}{2} - \frac{z}{12} + \cdots \\[5pt]
e^z &= 1 + z + \frac{1}{2} z^2 + \cdots
\end{aligned}
\]
to obtain the following expression for $e^z g(e^z)$
\[
\left(-\frac{1}{w_0 z} + \frac{1}{2} - \cdots \right)
\left(-\frac{1}{w_1 z} + \frac{1}{2} - \cdots \right)
\left(-\frac{1}{w_2 z} + \frac{1}{2} - \cdots \right)
\left(-dz + \frac{1}{2} d^2 z^2 - \cdots\right).
\]
From here one sees that the coefficient of $z^{-1}$ is
\[
- \frac{1}{2}d^2 \frac{1}{w_0 w_1 w_2} - \frac{1}{2} d \frac{1}{w_1 w_2}
- \frac{1}{2} d \frac{1}{w_0 w_2} - \frac{1}{2} d \frac{1}{w_0 w_1}
\]
and then the residue of $g(t)$ at $t=1$ is
\[
\Res(g(t),t=1) = - \frac{d (d+w_0+w_1+w_2)}{2 w_0 w_1 w_2} = - \frac{d(d+|w|)}{2\bar{w}}.
\]

\subsection{Residue at \texorpdfstring{$t=\zeta^i_{w_\ell} \neq 1$}{t=zeta-i-l-nonzero}}

In this case $t=\zeta_{w_\ell}^i \neq 1$ is a simple pole of $g(t)$ and the limit of
$(t-\zeta_{w_\ell}^i) g(t)$ when $t$ tends to~$\zeta_{w_\ell}^i$ computes the corresponding residue.
To fix the ideas assume for instance that $\ell=2$. Then
\[
\Res(g(t),t=\zeta_{w_2}^i) =
\lim_{t \to \zeta_{w_2}^i} \frac{(t-\zeta_{w_2}^i) (t^{-d}-1)}{(1-t^{w_0})(1-t^{w_1})(1-t^{w_2})t} =
\frac{1}{w_2} \cdot \frac{1-\zeta_{w_2}^{-id}}{(1-\zeta_{w_2}^{iw_0})(1-\zeta_{w_2}^{iw_1})}.
\]
Analogously one obtains the residues at $t = \zeta_{w_\ell}^i$ for $\ell = 0,1$.

\subsection{Summary}

For better presentation of the formula we have obtained so far for $\Eh_{\w}(d)$, we need to introduce some notation,
we set
\begin{equation}\label{eq:R-root-unity}
R_{(w_2;\, w_0,w_1)}(d) := 
- \frac{1}{w_2} \sum_{i=1}^{w_2-1} \frac{1-\zeta_{w_2}^{-id}}{(1-\zeta_{w_2}^{iw_0})(1-\zeta_{w_2}^{iw_1})}
\end{equation}
and analogously one obtains formulas for $R_{(w_0;\, w_1,w_2)}(d)$ and $R_{(w_1;\, w_0,w_2)}(d)$. Also, the sum of the three terms
is denoted by
\begin{equation}\label{eq:RR-3terms}
R_w(d) := R_{(w_0;\, w_1,w_2)}(d) + R_{(w_1;\, w_0,w_2)}(d) + R_{(w_2;\, w_0,w_1)}(d)
\end{equation}
so that one has the compact formula
\begin{equation}\label{eq:numerical-RR}
\Eh_{\w}(d) = 1 + \frac{d(d+|w|)}{2\bar{w}} + R_{w}(d).
\end{equation}
A similar formula has been obtained in~\cite{BR07}.

We will see later that the residue of $g(z)$ at $t=1$, i.e~$\frac{d(d+|w|)}{2\bar{w}}$, can be understood as the intersection
number of two divisors in the weighted projective plane $\PP_w^2$, while the residue at $t = \zeta_{w_\ell}^i \neq 1$, i.e.~$R_w(d)$
has to do with an extra term that appears in the Riemann-Roch formula as a consequence of the fact that $\PP^2_w$
has three isolated singular points.


\section{Euler characteristic of a sheaf and intersection theory in \texorpdfstring{$\PP^2_w$}{P2w}}\label{sec:Euler-intersection}

The main goal of this section is to give a geometrical interpretation of some terms appearing
in~\eqref{eq:numerical-RR}. We will show that $\Eh_w(d)$ is the Euler characteristic
of a sheaf and the term $\frac{d(d+|w|)}{2\bar{w}}$ can be seen as the intersection number of two divisors
in the weighted projective plane. Throughout the discussion the virtual genus of a curve in $\PP^2_w$ will appear.
The study of the correction term $R_w(d)$ is postponed until section~\ref{sec:DELTA}.

We split this section in three different parts.

\subsection{The weighted projective plane}\label{sec:weight-proj-plane}

For a more detailed presentation we refer to~\cite{AMO14a,AMO14b}, cf.~\cite{Dol82}.

Let $w = (w_0,w_1,w_2) \in \ZZ^3_{\geq 1}$ be a weight vector. The weighted projective plane associated with $w$
is defined by
\[
\PP^2_w := \frac{\CC^3 \setminus \{0\}}{\sim},
\]
where $(x',y',z') \sim (x,y,z)$, if there exists $t \in \CC^{*}$ such that $x' = t^{w_0} x$, $y' = t^{w_0} y$,
and $z' = t^{w_0} z$. The class of $(x,y,z)$ is denoted by $[x:y:z]_{w}$ and we drop the subindex $w$ if no confusion
arises.

The weighted projective plane is an orbifold that can be covered by three charts $\PP^2_w = U \cup V \cup W$ where
$U = \{x \neq 0 \}$, $V = \{y\neq 0 \}$, and $W = \{z \neq 0\}$. The first chart is given by
\[
X(w_2;w_0,w_1) \longrightarrow W, \ [(x,y)] \mapsto [x:y:1]_w,
\]
where $X(w_2;w_0,w_1) = \CC^2/C_{w_2}$, $C_{w_2}$ denotes the cyclic group of the $w_2$-roots of unity in~$\CC^{*}$,
and the action is defined by $\xi \cdot (x,y) = (\xi^{w_0} x, \xi^{w_1}y)$.

Denote by $w_{ij} = \gcd(w_i,w_j)$ for $i,j = 0,1,2$, $i \neq j$ and put $v_i = \frac{w_i}{w_{ij}w_{ik}}$
for $\{i,j,k\} = \{0,1,2\}$. The following map
\[
\phi: \PP^2_w \longrightarrow \PP^2_v, \ [x:y:z]_w \mapsto [x^{w_{12}}:y^{w_{02}}:z^{w_{01}}]_v
\]
is an isomorphism of algebraic varieties and the weights $v_i$'s are pairwise coprime. From now on we will always assume
this condition on the weights $w_i$'s unless explicitly stated otherwise, see \S\ref{sec:non-coprime}.

\subsection{Intersection theory in \texorpdfstring{$\PP_w^2$}{Pw2}}

Again we cite \cite{AMO14a, AMO14b} for a more detailed exposition, see also~\cite{Ful98}.

In the context of intersection theory \cite{Mum61} there is a version of Bézout's theorem in the weighted projective plane.
Let $D_1$ and $D_2$ be two Weil divisors in $\PP^2_w$, then the intersection multiplicity $D_1 \cdot D_2$ is well defined
and it verifies
\[
D_1 \cdot D_2 = \frac{1}{w_0 w_1 w_2} \deg_w (D_1) \deg_w (D_2) = \frac{1}{\bar{w}} \deg_w (D_1) \deg_w (D_2),
\]
where $\deg_w (D_i)$ is the degree of $D_i$. Note that if $D_i$ is given by a quasihomogeneous polynomial $H_i$,
then $\deg_w (D_i)$ is simply the degree of $H_i$ as a quasihomogeneous polynomial, or equivalently, the degree
of $H_i(x^{w_0},y^{w_1},z^{w_2})$.

The canonical divisor $K_{\PP^2_w}$ is the class of minus the sum of the three axes. Then it has degree $-w_0-w_1-w_2 = -|w|$.
Hence
\begin{equation}\label{eq:12DDK}
\frac{1}{2} D \cdot (D - K_{\PP^2_w}) = \frac{d(d+|w|)}{2\bar{w}}.
\end{equation}
This is the second term on the right-hand side of equation~\eqref{eq:numerical-RR}.

Shifting by the canonical divisor one gets $g_w(d)$ the \emph{virtual genus} of a curve of degree $d$ in $\PP^2_w$,
that is,
\begin{equation}\label{eq:virtual-genus}
g_w(d) := \chi(\PP^2_w,\cO_{\PP^2_w}) + \frac{1}{2} D \cdot (D+K_{\PP^2_w}) = 1 + \frac{d(d-|w|)}{2\bar{w}}.
\end{equation}

\subsection{Euler characteristic of \texorpdfstring{$\cO_{\PP_w^2}(D)$}{OPw2D}}

Each solution $(i,j,k) \in \cD_{\w,d}\cap \ZZ^3$ gives rise to a monomial $x^i y^j z^k$ of weighted degree~$d$.
This way one finds a basis as a $\CC$-vector space of $\CC[x,y,z]_{w,d}$, the quasihomogeneous polynomials with respect
to $w$ of degree $d$. It turns out that this vector space is isomorphic to the cohomology $H^0(\PP^2_w, \cO_{\PP^2_w}(D))$
where $D$ is any Weil divisor in the weighted projective plane of degree $d$ and
\[
\cO_{\PP^2_w}(D) = \{ f \in \cK_{\PP^2_w} \mid (f) + D \geq 0 \}
\]
being $\cK_{\PP^2_w}$ the sheaf of rational function on $\PP^2_w$. If $D = \{ H = 0 \} \geq 0$ is an effective divisor
of degree $d$, then the isomorphism is given by
\[
\begin{tikzcd}[row sep=-9pt,/tikz/column 1/.append style={anchor=base east},/tikz/column 2/.append style={anchor=base west}]
{\CC[x,y,z]_{w,d}}\rar&H^0(\PP^2_w, \cO_{\PP^2_w}(D))\\
F\rar[mapsto]&\dfrac{F}{H}.
\end{tikzcd}
\]
By Serre's duality $H^2(\PP^2_w, \cO_{\PP^2_w}(D))$ is isomorphic to $H^0(\PP^2_w, \cO_{\PP^2_w}(K_{\PP^2_w}-D))$
where $K_{\PP^2_w}$ is the canonical divisor of $\PP^2_w$. Since $K_{\PP^2_w}-D$ has negative degree,
namely $-|w|-d = -(d+w_0+w_1+w_2)$, these cohomology groups vanish. On the other hand,
$H^1(\PP^2_w, \cO_{\PP^2_w}(D)) = 0$ always holds~\cite[\S 1.4]{Dol82}. Then the Euler characteristic $\chi(\PP^2_w, \cO_{\PP^2_w}(D))$
is concentrated in degree zero and one has
\begin{equation}\label{eq:euler-char-OPD}
\chi(\PP^2_w, \cO_{\PP^2_w}(D)) := \sum_{i=0}^{2} \dim H^i(\PP^2_w, \cO_{\PP^2_w}(D))
= \dim H^0(\PP^2_w, \cO_{\PP^2_w}(D)) = \Eh_{\w}(d).
\end{equation}
This way we have just given a geometrical interpretation to the left-hand side of equation~\eqref{eq:numerical-RR}.
Note that $\chi(\PP^2_w,\cO_{\PP^2_w}) = \Eh_{\w}(0) = 1$. This corresponds to the first term on right-hand side
of equation~\eqref{eq:numerical-RR}.


\section{The \texorpdfstring{$\Delta_P$}{DP}-invariant of a divisor}\label{sec:DELTA}

The purpose of this section is to study the correction term $R_w(d)$ from~\eqref{eq:RR-3terms} and~\eqref{eq:numerical-RR}.
We will show that each $R_{(w_i;w_j,w_k)}(d)$, where $\{i,j,k\} = \{0,1,2\}$, is a local invariant of a divisor in a
cyclic quotient singularity. This together with the results from \S\ref{sec:Euler-intersection} will lead us
to a new proof of Theorem~\ref{thm:RR-formula} that was already established in \cite{Bla95}, cf.~\cite{CM19}.

We start by defining a local invariant associated with a cyclic quotient singularity,
namely the $\Delta_P$-invariant.
Given $p,q,r\in \ZZ_{\geq 1}$ we define the following number which generalizes the combinatorial number $\binom{d}{2}$:
\begin{equation}\label{eq-delta-comb}
\delta^{(p,q)}_{r}:= \frac{r (qr-p-q+1)}{2p}.
\end{equation}
Note that $\binom{d}{2} = \delta^{(1,1)}_d$.
Consider also the following cardinality
\[
A^{(p,q)}_r:= \#\{ (i,j)\in \ZZ^2_{\geq 1} \mid  pi+qj\leq qr \}.
\]

\begin{dfn}\label{defDELTA}
Let $p,q \in Z_{\geq 0}$ be two coprime integers. Consider the action $C_p \times \CC^2 \to \CC^2$
given by $\xi \cdot (x,y) = (\xi^{-1}x,\xi^{q}y)$ where $C_p = \{ \xi \in \CC^{*} \mid \xi^{p}=1 \}$.
This quotient space is denoted by $X(p;-1,q)=X$. Let $k\geq 0$ and $P\in X(p;-1,q)$.
The \emph{$\Diff_P$-invariant} of $X$ is defined as follows
\[
\Diff_{(p;-1,q)}(k):= A^{(p,q)}_r-\delta^{(p,q)}_r,
\]
where $r=q^{-1}k \mod p$.
\end{dfn}

\begin{remark}\label{remark:DELTA}
Assume $w = (w_0,w_1,w_2)$ and the weights $w_i$'s are pairwise coprime.
In order to compute the $\Delta_P$-invariant for a general cyclic quotient space
one uses the following relation
$\Delta_{(w_2;w_0,w_1)}(d) = \Delta_{(w_2;-1,-w_0^{-1}w_1 \mod w_2)}(-w_0^{-1}d \mod w_2)$.
\end{remark}

For cyclic quotient singularities the $\Delta_X$-invariant has an intrinsic geometric meaning,
see~\eqref{eq:DeltaXk}.
We will show that the $\Delta_P$-invariant is related to the correction term $R_w(d)$,
see Proposition~\ref{prop:Delta-R}. Before that we need to prove two technical results.

\begin{lemma}\label{lemma:LgR12}
Let $p,q,r \in \ZZ^3_{\geq 1}$. Then one has
\begin{enumerate}[label=\rm(\arabic{enumi})]
\item $\Eh_{(p,q,1)}(qr-p-q) = A^{(p,q)}_r$,
\item $g_{(p,q,1)}(qr+1) = \delta^{(p,q)}_r - \frac{p+q}{2pq} + 1$.
\end{enumerate}
\end{lemma}

\begin{proof}
These two formulas easily hold from the definitions of $\Eh_{\w}(d)$, $g_{\w}(d)$, $A^{(p,q)}_r$, and $\delta^{(p,q)}_r$
as follows
\[
\begin{aligned}
\Eh_{(p,q,1)}(qr-p-q) &= \# \{ (i,j,k) \in \ZZ_{\geq 0}^3 \mid ip + jq + k = qr-p-q \} \\
&= \# \{ (i,j) \in \ZZ_{\geq 0}^2 \mid ip + jq \leq qr-p-q \} \\
&= \# \{ (i,j) \in \ZZ_{\geq 0}^2 \mid (i+1)p + (j+1)q \leq qr \} = A^{(p,q)}_r, \\[10pt]
g_{(p,q,1)}(qr+1) &= 1 + \frac{(qr+1)(qr-p-q)}{2pq} = 1 + \frac{qr(qr-p-q)}{2pq} + \frac{qr-p-q}{2pq} \\
&= 1 + \frac{qr(qr-p-q+1)}{2pq} + \frac{\cancel{qr}-p-q}{2pq} - \cancel{\frac{qr}{2pq}} = \delta^{(p,q)}_r - \frac{p+q}{2pq} + 1,
\end{aligned}
\]
as it was claimed in the statement.
\end{proof}

\begin{lemma}\label{lemma:LgR}
Let $w = (w_0,w_1,w_2)$ be a weight vector with $w_i$'s pairwise coprime.
Consider $p = w_2$, $q = (- w_0^{-1} w_1 \mod w_2)$, and $r = (w_1^{-1} d \mod w_2) = (-(qw_0)^{-1}d \mod w_2)$.
Then
\[
R_{(p,q,1)}(qr-p-q) = - R_{(w_2;w_0,w_1)}(d-|w|) + \frac{p+q}{2pq} - 1.
\]
\end{lemma}

\begin{proof}
By definition the correction term splits in three terms
\begin{equation}\label{eq:R3terms}
R_{(p,q,1)}(qr-p-q) = R_{(p;q,1)}(qr-p-q) + R_{(q;p,1)}(qr-p-q) + R_{(1;p,q)}(qr-p-q).
\end{equation}
We plan to calcute these three terms separately.
The third term in~\eqref{eq:R3terms} is clearly zero.

For the second one, let us consider $\zeta_q$ a primitive $q$th root of unity. Then
\begin{equation}\label{eq:Rqp1}
R_{(q;p,1)}(qr-p-q) = R_{(q;p,1)}(-p) =
- \frac{1}{q} \sum_{i=1}^{q-1} \frac{\cancel{\phantom{(}1-\zeta_q^{ip}\phantom{)}}}{\cancel{(1-\zeta_q^{ip})}(1-\zeta_q^i)} =
- \frac{q-1}{2q}.
\end{equation}
For the last calculation in~\eqref{eq:Rqp1} see \S\ref{sec:formulaq} in the appendix.

For the first term in~\eqref{eq:R3terms}, let us fix $\zeta_{w_2}$ a primitive $p$th root of unity. By definition
\[
R_{(p;q,1)} (qr-p-q) = R_{(p;q,q)} (qr-q) = -\frac{1}{p} \sum_{i=1}^{p-1} \frac{1-\zeta_{w_2}^{-i(qr-q)}}{(1-\zeta_{w_2}^{iq})(1-\zeta_{w_2}^i)}.
\]
Since $w_2$ and $w_0$ are coprime we can substitute $\zeta_{w_2}$ by $\zeta_{w_2}^{-w_0}$ and the result remains
\[
-\frac{1}{w_2} \sum_{i=1}^{w_2-1} \frac{1-\zeta_{w_2}^{iqrw_0-iqw_0}}{(1-\zeta_{w_2}^{-iw_0q})(1-\zeta_{w_2}^{-iw_0})} =
-\frac{1}{w_2} \sum_{i=1}^{w_2-1} \frac{1-\zeta_{w_2}^{-id+iw_1}}{(1-\zeta_{w_2}^{iw_1})(1-\zeta_{w_2}^{-iw_0})}.
\]
Multiplying numerator and denominator by $\zeta_{w_2}^{iw_0}$ one obtains
\begin{equation}\label{eq:Rpq1}
R_{(p;q,1)} (qr-p-q) =
-\frac{1}{w_2} \sum_{i=1}^{w_2-1} \frac{\zeta_{w_2}^{-id+iw_0+iw_1}-\zeta_{w_2}^{iw_0}}{(1-\zeta_{w_2}^{iw_1})(1-\zeta_{w_2}^{iw_0})}.
\end{equation}
On the other hand
\begin{equation}\label{eq:Rw2w0w1}
R_{(w_2;w_0,w_1)} (d-|w|) =
-\frac{1}{w_2} \sum_{i=1}^{w_2-1} \frac{1-\zeta_{w_2}^{-i(d-w_0-w_1)}}{(1-\zeta_{w_2}^{iw_0})(1-\zeta_{w_2}^{iw_1})}.
\end{equation}
Equations~\eqref{eq:Rpq1} and~\eqref{eq:Rw2w0w1} provide
\begin{equation*}
R_{(p;q,1)} (qr-p-q) + R_{(w_2;w_0,w_1)} (d-|w|) =
-\frac{1}{w_2} \sum_{i=1}^{w_2-1} \frac{\cancel{1-\zeta_{w_2}^{iw_0}}}{\cancel{(1-\zeta_{w_2}^{iw_0})}(1-\zeta_{w_2}^{iw_1})} = 
- \frac{w_2-1}{2w_2},
\end{equation*}
see \S\ref{sec:formulaq} in the appendix.

Finally observe that summing up all the contributions to $R_{(p,q,1)}(qr-p-q)$ from~\eqref{eq:R3terms} yields the desired formula.
\end{proof}

\begin{prop}\label{prop:Delta-R}
$\Delta_{(w_2;w_0,w_1)}(d) = - R_{(w_2;w_0,w_1)}(d-|w|)$.
\end{prop}

\begin{proof}
Shifting~\eqref{eq:numerical-RR} by $-|w|$ and recalling the definition of $g_w(d)$ from~\eqref{eq:virtual-genus},
one obtains the equivalent formula
\begin{equation}\label{eq:numerical-RR2}
\Eh_{\w}(d-|w|) = g_w(d) + R_w(d-|w|).
\end{equation}
Let us consider $p = w_2$, $q = (- w_0^{-1} w_1 \mod w_2)$, and $r = (w_1^{-1} d \mod w_2) = (-(qw_0)^{-1}d \mod w_2)$
as in Lemma~\ref{lemma:LgR}. Substituting $w$ by $(p,q,1)$ and $d$ by $qr+1$ in formula~\eqref{eq:numerical-RR2},
one obtains
\[
\Eh_{\w}(qr-p-q) = g_{(p,q,1)}(qr+1) + R_{(p,q,1)}(qr-p-q).
\]
Then Lemma~\ref{lemma:LgR} yields
\[
A^{(p,q)}_r = \delta^{(p,q)}_r - R_{(w_2;w_0,w_1)}(d-|w|).
\]
Finally the claim follows from the observation
\[
\Delta_{(w_2;w_0,w_1)}(d) = \Delta_{(p;-1,q)}(qr) = A^{(p,q)}_r - \delta^{(p,q)}_r,
\]
see Definition~\ref{defDELTA} and Remark~\ref{remark:DELTA}.
\end{proof}

Now we are ready to proof the main result of this paper.

\begin{proof}[Proof of Theorem{\rm~\ref{thm:RR-formula}}]
Let us consider the expression~\eqref{eq:numerical-RR},
\[
\Eh_{\w}(d) = 1 + \frac{d(d+|w|)}{2\bar{w}} + R_{w}(d).
\]
Recall that $\Eh_{\w}(d) = \chi(\PP^2_w,\cO_{\PP^2_w}(D))$ from~\eqref{eq:euler-char-OPD}
and $\frac{d(d+|w|)}{2\bar{w}} = \frac{1}{2} D \cdot (D-K_{\PP^2_w})$ from~\eqref{eq:12DDK}.
By definition, see~\eqref{eq:RR-3terms},
\[
R_{w}(d) = R_{(w_0;\, w_1,w_2)}(d) + R_{(w_1;\, w_0,w_1)}(d) + R_{(w_2;\, w_0,w_1)}(d).
\]
Finally Proposition~\ref{prop:Delta-R} allows us to rewrite these three addends in terms of
the $\Delta_P$-invariant, namely $R_{(w_k;w_i,w_j)}(d) = - \Delta_{(w_k;w_i,w_j)}(d+|w|)$.
Moreover, each term corresponds to a singular point of $\PP^2_w$, see \S\ref{sec:weight-proj-plane}.
Now the proof is complete.
\end{proof}

As a consequence of this study, in the following two corollaries, we obtain some properties of the local
and global correction terms.

\begin{cor}\label{cor:global-duality}
Let $D$ be a divisor in $\PP^2_w$ and consider $R_{\PP^2_w}(D)$ be the correction term in the Riemann-Roch
formula such that
\[
\chi(\PP^2_w, \cO_{\PP^2_w}(D)) = 1 + \frac{1}{2} D \cdot (D - K_{\PP^2_w}) + R_{\PP^2_w}(D).
\]
Then $R_{\PP^2_w}(D) = R_{\PP^2_w}(K_{\PP^2_w}-D)$.
\end{cor}

\begin{proof}
By Serre's duality $\chi(\PP^2_w, \cO_{\PP^2_w}(D)) = \chi(\PP^2_w, \cO_{\PP^2_w}(K_{\PP^2_w}-D))$.
Note that the term $\frac{1}{2} D \cdot (D - K_{\PP^2_w})$ also remains invariant after substituting
$D$ by $K_{\PP^2_w}-D$. Then the same happens for $R_{\PP^2_w}(D)$ and the claim follows.
\end{proof}

\begin{cor}
Take $X = X(w_2;w_0,w_1)$ and denote by $|w|=w_0+w_1+w_2$. Then the following holds:
\begin{enumerate}[label=\rm(\arabic{enumi})]
\item $R_X(d) = R_X(-|w|-d)$ and $\Delta_X(d) = \Delta_X(|w|-d)$,
\item $R_X(d) = - \Delta_X(-d)$,
\item $R_X(-|w|) = \Delta_X(|w|) = 0$.
\end{enumerate}
\end{cor}

\begin{proof}
It is enough to prove $R_X(d) = R_X(-|w|-d)$. The rest of the formulas follows from Proposition~\ref{prop:Delta-R}
and the fact that $R_X(0) = \Delta_X(0) = 0$ by definition.

Assume without loss of generality that $X = X(p;q,1)$ with $\gcd(p,q)=1$. Consider $r = (q^{-1}d+1 \mod p)$
so that $R_X(d) = R_X(qr-q)$ and $R_X(-|w|-d) = R_X(-qr-1)$. Now we use the duality of the global
correction term to show the duality for~$R_X$.

By Corollary~\ref{cor:global-duality}, $R_{(p,q,1)}(qr-p-q) = R_{(p,q,1)}(-qr-1)$.
By definition
\[
\begin{aligned}
R_{(p,q,1)}(qr-p-q) &= R_{(p;q,1)}(qr-q) + R_{(q;p,1)}(-p), \\[2.5pt]
\rotatebox{90}{$=$} \hspace{35pt} \\
R_{(p,q,1)}(-qr-1) &= R_{(p;q,1)}(-qr-1) + R_{(q;p,1)}(-1).
\end{aligned}
\]
One can show as in~\eqref{eq:Rqp1} that $R_{(q;p,1)}(-p) = R_{(q;p,1)}(-1) = \frac{q-1}{2q}$,
cf.~\S\ref{sec:formulaq}, and thus we obtain $R_{(p;q,1)}(qr-q) = R_{(p;q,1)}(-qr-1)$ as required.
\end{proof}


\section{Effective computation of the correction term}\label{sec:effective-correction-term}

In this section we will develop an algorithm to efficiently compute the correction term $R_{(d;a,b)}(k)$
or equivalently $\Delta_{(d;a,b)}(k)$. In particular, we will show Theorem~\ref{thm:correction-term}.
We need some preliminary results.

Consider a Weil divisor $D$ in $\PP^2_w$ of degree $d = \deg_w (D) \in \ZZ$.
Recall that
\[
h^0(\PP^2_w,\cO_{\PP^2_w}(D)) =
\begin{cases}
\Eh_{\w}(d) & \text{if $d \geq 0$}, \\
0 & \text{if $d < 0$}.
\end{cases}
\]
By Serre's duality,
\[
h^2(\PP^2_w,\cO_{\PP^2_w}(D)) =
h^0(\PP^2_w,\cO_{\PP^2_w}(K_{\PP^2_w}-D)) =
\begin{cases}
\Eh_{\w}(-|w|-d) & \text{if $d \leq -|w|$}, \\
0 & \text{if $d > -|w|$}.
\end{cases}
\]
On the other hand, $h^1(\PP^2_w,\cO_{\PP^2_w}(D))$ vanishes for all $d \in \ZZ$, \cite[\S 1.4]{Dol82}.
Then we have just proven that
\begin{equation}\label{eq:vanishing-chi}
\chi(\PP_w^2, \cO_{\PP^2_w}(D)) = 0, \quad -|w| < \deg_w(D) < 0,
\end{equation}
see Figure~\ref{fig:vanishing-hi}.
\begin{figure}[ht]
\centering
\begin{tikzpicture}
\begin{footnotesize}
\draw[fill=lightgray] (-4,0.0) rectangle node[]{$\chi=0$} (-1,0.5);
\draw[very thick] (-7,0) -- (2,0) node[right]{$\ZZ$};
\draw (-5,-0.1) node[below]{\hspace{-12pt}$-|w|$} -- (-5,0.1);
\draw (-4,0) node[below]{$-|w|+1$};
\draw (-1,0) node[below]{$-1$};
\draw (0,-0.1) node[below]{$0$} -- (0,0.1);
\draw[dashed,->] (0,0) -- (0,1.5) -- (-6,1.5) node[left]{$h^0=0$};
\draw[dashed,->] (-5,0) -- (-5,1) -- (1,1) node[right]{$h^2=0$};
\draw[<->] (-6,2) -- node[above]{$h^1=0$} (1,2);
\end{footnotesize}
\end{tikzpicture}
\caption{Vanishing of $H^i(\PP^2_w,\cO_{\PP^2_w}(D))$, $i=0,1,2$.}
\label{fig:vanishing-hi}
\end{figure}
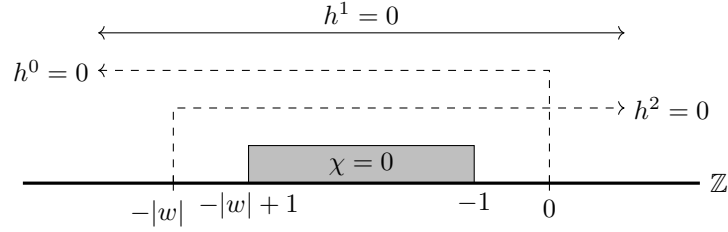

One of the strategies of this section is to use birational morphisms and resolution of singularities. Even if $D$ is a Weil divisor,
its pull-back is in general a $\QQ$-divisor. Also the relative canonical divisor of a morphism is a $\QQ$-divisor,
see~\cite{Sak84}.
To deal with $\QQ$-divisors we need to introduce some notation.

\begin{dfn}\label{dfn:floor-fpart}
For a $\QQ$-divisor $D$ we write $D = \floor{D} + \fpart{D}$, where $\fpart{D}$ is the fractional part of $D$
and $\floor{D}$ is the integral part of $D$. If $D = \sum_i a_i D_i$ is the decomposition of $D$ into prime divisors,
then we have
\[
\floor{D} = \sum_i \floor{a_i} D_i, \qquad
\fpart{D} = \sum_i \fpart{a_i} D_i,
\]
where $\floor{a_i} \in \ZZ$, $a_i-1 < \floor{a_i} \leq a_i$, $a_i = \floor{a_i} + \fpart{a_i}$, $0 \leq \fpart{a_i} < 1$.
\end{dfn}

\begin{lemma}\label{lemma:fpart}
Let $d,q \in \ZZ$ be two coprime integers with $d \geq 1$. Then,
\[
R_{(d;1,-q)}(k) + R_{(d;1,q)}(k) + \fpart{\frac{k}{d}} = 0,
\quad \forall k \in \ZZ.
\]
\end{lemma}

\begin{proof}
Assume $k$ is not a multiple of $d$, otherwise the statement clearly follows.
Let us first consider the case $-d < k < 0$, below we will show the statement for a general $k \in \ZZ$.
To simplify notation, denote by $p=d-q$ so that $-q \equiv p \mod d$.
Consider three weight vectors $w_1 = (d,p,1)$, $w_2 = (d,1,q)$, $w_3 = (1,p,q)$.
Using~\eqref{eq:vanishing-chi}, since $-d < k < 0$, one has
\[
\begin{aligned}
0 &= \chi(\PP_{w_1}^2, \cO_{\PP^2_{w_1}}(k)) = g_{w_1}(k+d+p+1) + R_{(d;1,p)}(k) + R_{(p;1,d)}(k), \\
0 &= \chi(\PP_{w_2}^2, \cO_{\PP^2_{w_2}}(k)) = g_{w_2}(k+d+1+q) + R_{(d;1,q)}(k) + R_{(q;1,d)}(k), \\
0 &= \chi(\PP_{w_3}^2, \cO_{\PP^2_{w_3}}(k)) = g_{w_3}(k+1+p+q) + R_{(p;1,q)}(k) + R_{(q;1,p)}(k).
\end{aligned}
\]
Note that $R_{(p;1,d)}(k) = R_{(p;1,q)}(k)$ and $R_{(q;1,d)}(k) = R_{(q;1,p)}(k)$, since $d = p + q$.
Then subtracting the third equation from the sum of the first two, one obtains
\[
R_{(d;1,p)}(k) + R_{(d;1,q)}(k) + g_{w_1}(k+d+p+1) + g_{w_2}(k+d+1+q) - g_{w_3}(k+1+p+q) = 0.
\]
A straightforward computation provides
\[
g_{w_1}(k+d+p+1) + g_{w_2}(k+d+1+q) - g_{w_3}(k+1+p+q) = 1 + \frac{k}{d} = \fpart{\frac{k}{d}}.
\]
For a general $k \in \ZZ$ with $k \not\equiv 0 \mod d$, consider the Euclidean division $k = cd + r$, $0 < r < d$.
Hence $-d < r-d < 0$ and one can apply the statement for $k' := r-d$. The result follows from the fact that
$R_{(d;1,p)}(k') = R_{(d;1,p)}(k)$, $R_{(d;1,q)}(k') = R_{(d;1,q)}(k)$, and~$\fpart{\frac{k'}{d}} = \fpart{\frac{k}{d}}$.
\end{proof}

The following result is a generalization of \cite[\S 1.2]{Bla95} when the morphism $\pi: \tilde{X} \to X$ is not necessarily
a resolution of $X$. In order to prove Theorem~\ref{thm:correction-term}, we will use this result for partial resolutions.

\begin{prop}\label{prop:correction-term}
Let $\pi: \tilde{X} \to X$ be a birational morphism between two (not necessarily smooth) projective algebraic surfaces
and consider $D$ a Weil divisor in $X$.
Then
\[
R_X(D) = - \frac{1}{2} \fpart{ \pi^* D } \cdot ( \floor{\pi^* D} - K_{\tilde{X}} ) + R_{\tilde{X}}(\floor{\pi^* D}).
\]
\end{prop}

\begin{proof}
The key strategy is to use the projection formula from \cite[Theorem 2.1]{Sak84}
that implies $\chi(X,\cO_X(D)) =
\chi(\tilde{X},\cO_{\tilde{X}}(\floor{\pi D}))$. By definition,
\begin{equation}\label{eq:chi-chi}
\left\{
\begin{aligned}
\chi(X,\cO_X(D)) &= 1 + \frac{1}{2} D \cdot (D - K_X) + R_X(D), \\[5pt]
\chi(\tilde{X},\cO_{\tilde{X}}(\floor{\pi^* D})) &=
1 + \frac{1}{2} \floor{\pi^* D} \cdot ( \floor{\pi^* D} - K_{\tilde{X}} ) + R_{\tilde{X}}(\floor{\pi^* D}).
\end{aligned}\right.
\end{equation}
The idea is to rewrite the second expression to find a relation between the corrections terms
$R_X(D)$ and $R_{\tilde{X}}(\floor{\pi^* D})$. First recall that $\floor{\pi^* D} = \pi^{*}D + \fpart{\pi^* D}$,
see Definition~\ref{dfn:floor-fpart},
\begin{equation}\label{eq:floor-pi-D}
\floor{\pi^* D} \cdot ( \floor{\pi^* D} - K_{\tilde{X}} ) =
\pi^* D \cdot ( \floor{\pi^* D} - K_{\tilde{X}} ) - \fpart{\pi^* D} \cdot ( \floor{\pi^* D} - K_{\tilde{X}} ).
\end{equation}
Also $K_{\tilde{X}} = \pi^* K_X + K_\pi$ and thus
\begin{equation}\label{eq:pi-D-long}
\begin{aligned}
\pi^* D \cdot ( \floor{\pi^* D} - K_{\tilde{X}} ) &=
\pi^* D \cdot ( \pi^* D - \fpart{\pi^* D} - \pi^* K_X - K_\pi ) \\
&= \pi^* D \cdot ( \pi^* D - \pi^* K_X ) \\
&= D \cdot ( D - K_X ),
\end{aligned}
\end{equation}
where we have used that $\pi^* D \cdot \fpart{\pi^* D} = \pi^* D \cdot K_\pi = 0$, since $\fpart{\pi^* D}$
and $K_{\pi}$ only have exceptional part, and $\pi^* D_1 \cdot \pi^* D_2 = D_1 \cdot D_2$ for any pair of Weil divisors
$D_1, D_2$ in $X$. Now \eqref{eq:chi-chi}, \eqref{eq:floor-pi-D}, \eqref{eq:pi-D-long} yield
\[
\chi(\tilde{X},\cO_{\tilde{X}}(\floor{\pi^* D})) = \chi(X,\cO_X(D))
- \frac{1}{2} \fpart{ \pi^* D } \cdot ( \floor{\pi^* D} - K_{\tilde{X}} )
+ R_{\tilde{X}}(\floor{\pi^* D}) - R_X(D).
\]
The claim follows from the projection formula as it was mentioned above.
\end{proof}

Now we are ready to show the second main result of this work.

\begin{proof}[Proof of Theorem{\rm~\ref{thm:correction-term}}]
Consider $\pi: \tilde{X} \to X = X(d;1,q)$ the weighted blowing-up at the origin with weights $(1,q)$.
Let $A$ be the divisor in $X$ given by $\{x=0\}$ and take $D = kA$.
Denote by $E$ the exceptional divisor of $\pi$ and $\widehat{A}$ the strict transform of $A$.
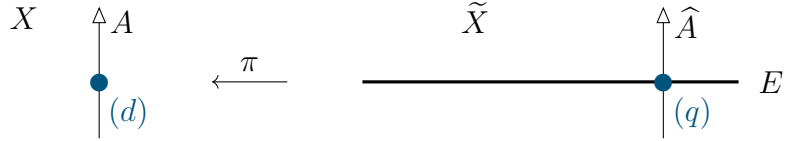
\begin{figure}[ht]
\centering
\begin{tikzpicture}[scale=1.0]
\definecolor{bluegreen2}{RGB}{0, 85, 127}
\colorlet{colsing}{bluegreen2}
\colorlet{colsingtext}{colsing}
\node at (-6,0.85) {$X$};
\draw[-open triangle 45] (-5, -0.75) -- (-5,1);
\node[color=colsing] (O1) at (-5,-0.02) {\Large $\bullet$};
\node[below, xshift=11pt, yshift=-1pt, text=colsingtext] at (O1) {$(d)$};
\node[right, yshift=23pt] at (O1) {$A$};
\draw[<-] (-3.5,0) -- node[above]{$\pi$} (-2.5,0);
\node at (0,0.85) {$\widetilde{X}$};
\draw[very thick] (-1.5,0) -- (3.5,0) node [right, xshift=3pt] {$E$};
\draw[-open triangle 45] (2.5, -0.75) -- (2.5,1);
\node[color=colsing] (O2) at (2.5,-0.02) {\Large $\bullet$};
\node[below, xshift=11pt, yshift=-1pt, text=colsingtext] at (O2) {$(q)$};
\node[right, yshift=23pt] at (O2) {$\widehat{A}$};
\end{tikzpicture}
\caption{Blowing-up at the origin of $X(d;1,q)$ with weights $(1,q)$.}
\label{fig:blowing-up}
\end{figure}
\break
Note that $\widehat{A}$ transversally intersects $E$ at a singular point of the ambient space of type
$(q;1,-d) = (q;1,-r)$ being $d = c \cdot q + r$, see Figure~\ref{fig:blowing-up}. One has
\[
\widehat{A} \cdot E = \frac{1}{q}, \qquad
E^2 = -\frac{d}{q}, \qquad
K_{\pi} = \left( \frac{1+q}{d} - 1 \right) E.
\]
Since $0 \leq k < d$, one gets
\[
\pi^* D = k \widehat{A} + \frac{k}{d} E
\quad \Longrightarrow \quad \floor{\pi^* D} = k \widehat{A},
\quad \fpart{\pi^* D} = \frac{k}{d} E. \\
\]
Therefore
\[
- \frac{1}{2} \fpart{\pi^* D} \cdot ( \floor{\pi^* D} - K_\pi ) =
-\frac{1}{2} \frac{k}{d} E \cdot \left( k \widehat{A} - \frac{1+q-d}{d} E \right) =
- \frac{k(k+1+q-d)}{2dq}.
\]
By Proposition~\ref{prop:correction-term},
\begin{equation}\label{eq:R-frac-R}
R_{(d;1,q)}(k) = - \frac{k(k+1+q-d)}{2dq} + R_{(q;1,-r)}(k).
\end{equation}
Finally Lemma~\ref{lemma:fpart} completes the proof.
\end{proof}

Note that equation~\eqref{eq:R-frac-R} already provides a recursive formula for computing $R_{(d;1,q)}(k)$
without using Lemma~\ref{lemma:fpart}. However, the corresponding algorithm for a quotient space of type $(d;1,d-1)$
have $d-1$ steps. Therefore the method has an exponential complexity in terms of size of the input, i.e.~the number
of the bits used to represent the input. By contrast, the worst case in the recursive algorithm provided by
Theorem~\ref{thm:correction-term} is given by the so-called Fibonacci sequence.


\section{Non-pairwise coprime case}\label{sec:non-coprime}

Assume the weights $w_0, w_1, w_2 \in \ZZ_{\geq 1}$ are not necessarily pairwise coprime
and denote by $w_{ij} = \gcd(w_i,w_j)$ for $i,j = 0,1,2$, $i \neq j$. Assume $\gcd(w_0,w_1,w_2) = 1$.
For a fixed $d \in \ZZ_{\geq 0}$,
let us choose a solution $(a_0,a_1,a_2)$ of
$\cD_{w,d} \cap \ZZ^{3} = \{ (i,j,k) \in \ZZ_{\geq 0}^3 \mid w_0 i + w_1 j + w_2 k = d \}$.
Consider the divisions
\[
\left\{
\begin{aligned}
a_0 &= q_0 w_{12} + r_0, \\
a_1 &= q_1 w_{02} + r_1, \\
a_2 &= q_2 w_{01} + r_2,
\end{aligned}
\right.
\]
where $0 \leq r_k < w_{ij}$ for all $\{i,j,k\} = \{0,1,2\}$. Denote by $v_i = \frac{w_i}{w_{ij}w_{ik}}$
for $\{i,j,k\} = \{0,1,2\}$ and by $e = v_0 q_0 + v_1 q_1 + v_2 q_2$. Then
\begin{equation}\label{eq:w0r0plus}
\begin{aligned}
w_0 r_0 +& w_1 r_1 + w_2 r_2 = w_0 (a_0 - q_0 w_{12}) + w_1 (a_1 - q_1 w_{02}) + w_2 (a_2 - q_2 w_{12}) \\
&= w_0 a_0 + w_1 a_1 + w_2 a_2 - w_0 q_0 w_{12} - w_1 q_1 w_{02} - w_2 q_2 w_{12} \\
&= d - \frac{w_0}{w_{01}w_{02}} q_0 w_{01}w_{02}w_{12} - \frac{w_1}{w_{01}w_{12}} q_1 w_{01}w_{02}w_{12}
- \frac{w_2}{w_{02}w_{12}} q_2 w_{01}w_{02}w_{12} \\
&= d - ( v_0 q_0 + v_1 q_1 + v_2 q_2 ) w_{01}w_{02}w_{12} = d - e w_{01}w_{02}w_{12}.
\end{aligned}
\end{equation}
If $(a'_0, a'_1, a'_2) \in \cD_{w,d} \cap \ZZ^{3}$ is another solution, then
\[
w_0 ( a'_0 - r_0 ) + w_1 ( a'_1 - r_1 ) + w_2( a'_2 - r_2 ) = e w_{01} w_{02} w_{12}.
\]
Since $\gcd(w_0,w_1,w_2) = 1$, the previous equation implies that $w_{ij}$ must divide $a'_k - r_k$,
that is, $a'_k \equiv r_k \mod w_{ij}$, for all $\{i,j,k\} = \{0,1,2\}$.
In particular, $a'_k \geq r_k$.
This shows that triple $(r_0,r_1,r_2)$ does not depend on the chosen $(a,b,c)$ and it is uniquely
determined by the initial data $w_0,w_1,w_2,d$. This motivates the following definition.

\begin{dfn}
Consider $w = (w_0, w_1, w_2) \in \ZZ_{\geq 1}^3$, $d \in \ZZ_{\geq 0}$, $(a_0,a_1,a_2) \in \cD_{w,d} \cap \ZZ^{3}$.
Then one defines
\[
r_{k,w}(d) = ( a_k \mod w_{ij} ) \in [0,w_{ij}),
\]
for any $\{i,j,k\} = \{0,1,2\}$.
\end{dfn}

The following map
\[
\begin{aligned}
\varphi: \cD_{w,d} \longrightarrow \cD_{v,e}, \
(i,j,k) \longmapsto \left( \frac{i-r_{0,w}(d)}{w_{12}}, \frac{j-r_{1,w}(d)}{w_{02}}, \frac{k-r_{2,w}(d)}{w_{01}} \right),
\end{aligned}
\]
where $w = (w_0,w_1,w_2)$ and $v=(v_0,v_1,v_2)$, is well defined and bijective. Therefore
\begin{equation}\label{eq:Lwd}
\Eh_{\w}(d) = \# (\cD_{w,d} \cap \ZZ^3) = \# (\cD_{v,e} \cap \ZZ^3) = \Eh_{v}(e)
\end{equation}
and the new weights $v_i$'s are pairwise coprime.

We summarize the previous discussion in the following result.

\begin{prop}
Let $w = (w_0,w_1,w_2) \in \ZZ_{\geq 1}^3$, $\gcd(w_0,w_1,w_2) = 1$, and $d \in \ZZ_{\geq 0}$.
For $\{i,j,k\}=\{1,2,3\}$ consider
\[
w_{ij} = \gcd(w_i,w_j), \quad
v_i = \frac{w_i}{w_{ij}w_{ik}}, \quad
e = \frac{d - (w_0 r_{0,w}(d) + w_1 r_{1,w}(d) + w_2 r_{2,w}(d))}{w_{01}w_{02}w_{12}}.
\]
Then $\Eh_w(d) = \Eh_v(e)$.

Moreover, if $\varepsilon, s_0, s_1, s_2 \geq 0$, then
$r_{k,w}(\varepsilon w_{01} w_{02} w_{12} + w_0 s_0 + w_1 s_1 + w_2 s_2) = s_k$ as long as
$0 \leq s_k < w_{ij}$ for all $\{i,j,k\} = \{0,1,2\}$. In such a case,
\[
\Eh_w(\varepsilon w_{01} w_{02} w_{12} + w_0 s_0 + w_1 s_1 + w_2 s_2 ) = \Eh_v (\varepsilon).
\]
\end{prop}

There is a geometric explanation of this phenomenon. Consider the isomorphism of algebraic varieties
defined by
\[
\phi: \PP^2_w \longrightarrow \PP^2_v, \ [x:y:z]_w \longmapsto [x^{w_{12}}:y^{w_{02}}:z^{w_{01}}]_v
\]
Choose an effective divisor $D$ of degree $d$ in $\PP^2_w$. It is given by a quasihomogeneous polynomial
of degree~$d$
\[
F = \sum_{(i,j,k) \in T_{w,d}} a_{ijk} x^i y^j z^k
= x^{r_0} y^{r_1} z^{r_2} \sum_{(i,j,k) \in T_{w,d}} a_{ijk} x^{i-r_0} y^{j-r_1} z^{k-r_2}.
\]
The corresponding divisor $\phi_*(D)$ in $\PP^2_v$ is given by
\begin{equation}\label{eq:polynomial-G}
G = x^{\frac{r_0}{w_{12}}} y^{\frac{r_1}{w_{02}}} z^{\frac{r_2}{w_{01}}}
\sum_{(i,j,k) \in T_{w,d}} a_{ijk} x^{\frac{i-r_0}{w_{12}}} y^{\frac{j-r_1}{w_{02}}} z^{\frac{k-r_2}{w_{01}}},
\end{equation}
that is,
\[
\phi_*D = \frac{r_0}{w_{12}} X + \frac{r_1}{w_{02}} Y + \frac{r_2}{w_{01}} Z + \tilde{D},
\]
where $X = \{x=0\}$, $Y = \{y=0\}$, $Z = \{z=0\}$, and $\tilde{D}$ is the integral
divisor in $\PP^2_v$ defined by the sum given in~\eqref{eq:polynomial-G}.
Using~\eqref{eq:w0r0plus} one sees that the $v$-degree of $\tilde{D}$ is $e$.
Note that $0 \leq \frac{r_k}{w_{ij}} < 1$ and this means that $\floor{\phi_*D} = \tilde{D}$.
Recall that
\[
\cO_{\PP^2_v}(\floor{\phi_*D})
= \{ f \in \cK_{\PP^2_v} \mid (f) + \floor{\phi_*D} \geq 0 \}
= \{ f \in \cK_{\PP^2_v} \mid \floor{(f) + \phi_*D} \geq 0 \}.
\]
The condition $\floor{(f) + \phi_*D} \geq 0$ is equivalent to $(f) + \phi_*D \geq 0$.
Then one has
\[
\cO_{\PP^2_v}(\phi_*D) = \cO_{\PP^2_v}(\floor{\phi_*D}) = \cO_{\PP^2_v}(\tilde{D}),
\]
cf.~\cite{Sak84}. Finally,
\[
\Eh_{\w}(d)
= \chi(\PP^2_w,\cO_{\PP^2_w}(D)) = \chi(\PP^2_v,\cO_{\PP^2_v}(\phi_*D))
= \chi(\PP^2_v,\cO_{\PP^2_v}(\tilde{D})) = \Eh_{v}(e).
\]
as it was claimed in~\eqref{eq:Lwd} above.

\begin{ex}
Let us consider $w = (1235, 6545, 2652)$ and $d = 1710721$. One checks that $(1, 106, 383)$
is a solution in $\cD_{w,d} \cap \ZZ^3$. Note that $w_{01} = 5$, $w_{02} = 13$, $w_{12} = 17$.
Then
\[
\begin{aligned}
r_{0,w}(d) &= (1 \mod 17) = 1, \\
r_{1,w}(d) &= (106 \mod 13) = 2, \\
r_{2,w}(d) &= (383 \mod 5) = 3.
\end{aligned}
\]
The new weights are
\[
v_0 = \frac{w_0}{w_{01}w_{02}} = 19, \quad
v_1 = \frac{w_1}{w_{01}w_{12}} = 77, \quad
v_2 = \frac{w_2}{w_{02}w_{12}} = 12.
\]
Finally,
\[
e = \frac{d - (w_0 r_0 + w_1 r_1 + w_2 r_2)}{w_{01}w_{02}w_{12}} = 1528.
\]
Then $\Eh_w(d) = \Eh_{(19,77,12)}(1528)$ and now the new weights are pairwise coprime.
We will see later in \S\ref{sec:overview} that the number of solutions of $T_{v,e} \cap \ZZ^3$ is precisely $70$.
\end{ex}


\section{Overview with a detailed example}\label{sec:overview}

Consider $w_0=19$, $w_1=77$, $w_2=12$, and $d_w=1528$. One manually checks that the number of nonnegative integers lying
on the plane $19 x + 77y + 12 z = 1528$ is $70$, that is, the cardinality of 
\[
\{ (i,j,k) \in \ZZ_{\geq 0}^3 \mid w_0 i + w_1 j + w_2 k = d_w \} \subset \RR^3
\]
is precisely $70$. Here is the list of all of them:
\begin{scriptsize}
\[
\begin{aligned}
& (0, 8, 76), (1, 9, 68), (2, 10, 60), (3, 11, 52), (4, 0, 121), (4, 12, 44), (5, 1, 113), (5, 13, 36), (6, 2, 105), (6, 14, 28), \\
& (7, 3, 97), (7, 15, 20), (8, 4, 89), (8, 16, 12), (9, 5, 81), (9, 17, 4), (10, 6, 73), (11, 7, 65), (12, 8, 57), (13, 9, 49), \\
& (14, 10, 41), (15, 11, 33), (16, 0, 102), (16, 12, 25), (17, 1, 94), (17, 13, 17), (18, 2, 86), (18, 14, 9), (19, 3, 78), (19, 15, 1), \\
& (20, 4, 70), (21, 5, 62), (22, 6, 54), (23, 7, 46), (24, 8, 38), (25, 9, 30), (26, 10, 22), (27, 11, 14), (28, 0, 83), (28, 12, 6), \\
& (29, 1, 75), (30, 2, 67), (31, 3, 59), (32, 4, 51), (33, 5, 43), (34, 6, 35), (35, 7, 27), (36, 8, 19), (37, 9, 11), (38, 10, 3), \\
& (40, 0, 64), (41, 1, 56), (42, 2, 48), (43, 3, 40), (44, 4, 32), (45, 5, 24), (46, 6, 16), (47, 7, 8), (48, 8, 0), (52, 0, 45), \\
& (53, 1, 37), (54, 2, 29), (55, 3, 21), (56, 4, 13), (57, 5, 5), (64, 0, 26), (65, 1, 18), (66, 2, 10), (67, 3, 2), (76, 0, 7).
\end{aligned}
\]
\end{scriptsize}

\noindent Figure~\ref{fig:example-1528} illustrates the set of solutions represented by little red (interior)
and green (boundary) point and the corresponding plane in $\RR^3$. Note that the drawing has been rescaled 
on the right-hand side for better viewing.

Note that Pick's theorem is far from being true. The number of points on the boundary is $b = 9$ and
the number of interior points is $i=61$. The area of such an imaginary triangle in $\RR^2$ would be
$A = i + \frac{b}{2} - 1 = 64.5$. However, the true area is $A \approx 5333.74$. This is not only because
the vertices are not integral but also because the triangle does not lie in the plane $\RR^2$.

\begin{figure}[th]
\centering
\tikzset{scaled unit vectors/.code={ \path (#1,0,0) coordinate (ex) (0,#1,0) coordinate (ey) (0,0,#1) coordinate (ez);},
    scaled cs/.style={x={(ex)},y={(ey)},z={(ez)}}}
\begin{tikzpicture}[scale=0.9,x={(-1cm,-1cm)},y={(1cm,0cm)}, z={(0cm,1cm)}]
\tikzset{scaled unit vectors=0.035}
\draw[scaled cs, ->, >=latex] (0,0,0) -- (95,0,0) node[above left]{$x$};
\draw[scaled cs, ->, >=latex] (0,0,0) -- (0,35,0) node[right]{$y$};
\draw[scaled cs, ->, >=latex] (0,0,0) -- (0,0,150) node[left]{$z$};
\draw[scaled cs, white, fill=blue, opacity=0.25] (1528/19,0,0) -- (0,1528/77,0) -- (0,0,1528/12) -- cycle;
\draw[scaled cs, blue, ultra thick] (1528/19,0,0) -- (0,1528/77,0) -- (0,0,1528/12) -- cycle;
\foreach \P in {(1, 9, 68), (2, 10, 60), (3, 11, 52), (4, 12, 44), (5, 1, 113), (5, 13, 36), (6, 2, 105), (6, 14, 28),
(7, 3, 97), (7, 15, 20), (8, 4, 89), (8, 16, 12), (9, 5, 81), (9, 17, 4), (10, 6, 73), (11, 7, 65), (12, 8, 57), (13, 9, 49),
(14, 10, 41), (15, 11, 33), (16, 12, 25), (17, 1, 94), (17, 13, 17), (18, 2, 86), (18, 14, 9), (19, 3, 78),
(19, 15, 1), (20, 4, 70), (21, 5, 62), (22, 6, 54), (23, 7, 46), (24, 8, 38), (25, 9, 30), (26, 10, 22), (27, 11, 14),
(28, 12, 6), (29, 1, 75), (30, 2, 67), (31, 3, 59), (32, 4, 51), (33, 5, 43), (34, 6, 35), (35, 7, 27),
(36, 8, 19), (37, 9, 11), (38, 10, 3), (41, 1, 56), (42, 2, 48), (43, 3, 40), (44, 4, 32), (45, 5, 24),
(46, 6, 16), (47, 7, 8), (53, 1, 37), (54, 2, 29), (55, 3, 21), (56, 4, 13), (57, 5, 5),
(65, 1, 18), (66, 2, 10), (67, 3, 2)}
    \draw [scaled cs, draw=black, fill=red] \P circle (2.5pt);
\foreach \P in {(0,8,76), (4,0,121), (16, 0, 102), (28, 0, 83), (40, 0, 64), (48, 8, 0), (52, 0, 45), (64, 0, 26), (76, 0, 7)}
    \draw [scaled cs, draw=black, fill=green] \P circle (2.5pt);
\end{tikzpicture}
\begin{tikzpicture}[scale=0.9,x={(-1cm,-1cm)},y={(8cm,0cm)}, z={(0cm,1cm)}]
\tikzset{scaled unit vectors=0.035}
\draw[scaled cs, ->, >=latex] (0,0,0) -- (95,0,0) node[above left]{$x$};
\draw[scaled cs, ->, >=latex] (0,0,0) -- (0,23,0) node[right]{$y$};
\draw[scaled cs, ->, >=latex] (0,0,0) -- (0,0,150) node[left]{$z$};
\filldraw [scaled cs, black] (1528/19,0,0) circle (3pt);
\draw[scaled cs, white, fill=blue, opacity=0.25] (1528/19,0,0) -- (0,1528/77,0) -- (0,0,1528/12) -- cycle;
\draw[scaled cs, blue, ultra thick] (1528/19,0,0) -- (0,1528/77,0) -- (0,0,1528/12) -- cycle;
\foreach \P in {(1, 9, 68), (2, 10, 60), (3, 11, 52), (4, 12, 44), (5, 1, 113), (5, 13, 36), (6, 2, 105), (6, 14, 28),
(7, 3, 97), (7, 15, 20), (8, 4, 89), (8, 16, 12), (9, 5, 81), (9, 17, 4), (10, 6, 73), (11, 7, 65), (12, 8, 57), (13, 9, 49),
(14, 10, 41), (15, 11, 33), (16, 12, 25), (17, 1, 94), (17, 13, 17), (18, 2, 86), (18, 14, 9), (19, 3, 78),
(19, 15, 1), (20, 4, 70), (21, 5, 62), (22, 6, 54), (23, 7, 46), (24, 8, 38), (25, 9, 30), (26, 10, 22), (27, 11, 14),
(28, 12, 6), (29, 1, 75), (30, 2, 67), (31, 3, 59), (32, 4, 51), (33, 5, 43), (34, 6, 35), (35, 7, 27),
(36, 8, 19), (37, 9, 11), (38, 10, 3), (41, 1, 56), (42, 2, 48), (43, 3, 40), (44, 4, 32), (45, 5, 24),
(46, 6, 16), (47, 7, 8), (53, 1, 37), (54, 2, 29), (55, 3, 21), (56, 4, 13), (57, 5, 5),
(65, 1, 18), (66, 2, 10), (67, 3, 2)}
    \draw [scaled cs, draw=black, fill=red] \P circle (2.5pt);
\foreach \P in {(0,8,76), (4,0,121), (16, 0, 102), (28, 0, 83), (40, 0, 64), (48, 8, 0), (52, 0, 45), (64, 0, 26), (76, 0, 7)}
    \draw [scaled cs, draw=black, fill=green] \P circle (2.5pt);
\filldraw [scaled cs, black] (1528/19,0,0) circle (3pt);
\filldraw [scaled cs, black] (0,1528/77,0) circle (3pt);
\filldraw [scaled cs, black] (0,0,1528/12) circle (3pt);
\draw[scaled cs] (1528/19,0,0) node[below right]{$80.42$};
\draw[scaled cs] (0,1528/77,0) node[below=5]{$19.84$};
\draw[scaled cs] (0,0,1528/12) node[above right]{$127.33$};
\end{tikzpicture}
\caption{Solution for $19i + 77j + 12k = 1528$. (True and rescaled picture)}
\label{fig:example-1528}
\end{figure}
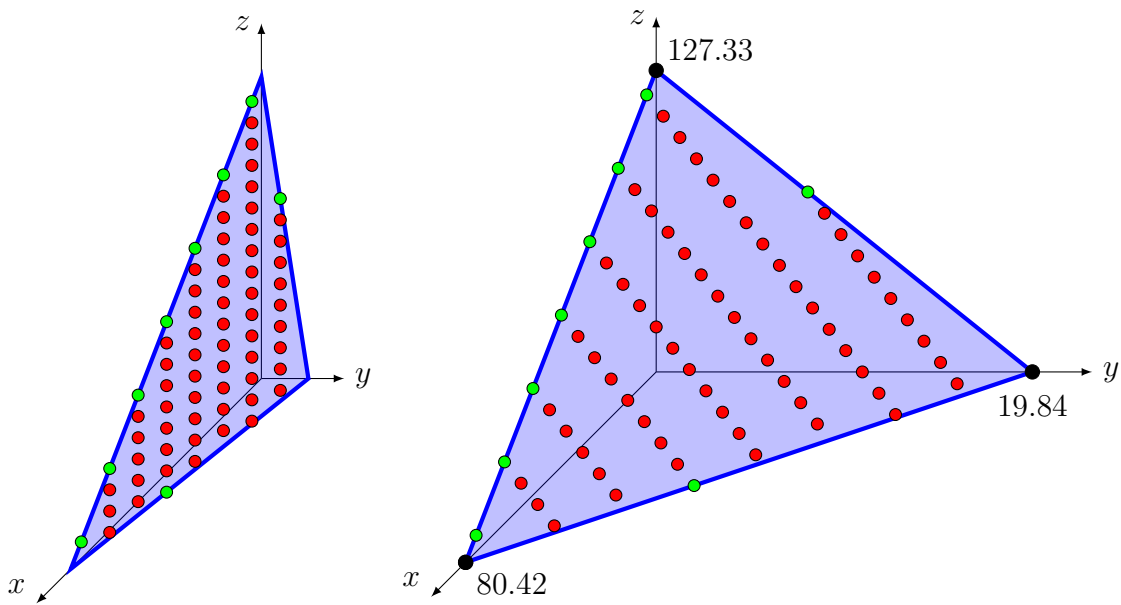

\subsection{Intersection theory and Riemann-Roch formula}

According to the theory each solution $(i,j,k) \in \ZZ^3$ gives rise to a monomial $x^i y^j z^k$ of weighted degree $d_w$.
This way one finds a basis as a $\CC$-vector space of $\CC[x,y,z]_{w,d}$, the quasihomogeneous polynomials with respect
to $w$ of degree $d$. It turns out that this vector space is isomorphic to the cohomology $H^0(\PP^2_w, \cO_{\PP^2_w}(D))$
where $D$ is any divisor in the weighted projective plane of degree $d_w$ and
\[
\cO_{\PP^2_w}(D) = \{ f \in \cK_{\PP^2_w} \mid (f) + D \geq 0 \}
\]
being $\cK_{\PP^2_w}$ the sheaf of rational functions on $\PP^2_w$. If $D = \{ H = 0 \} \geq 0$ is an effective divisor
of degree $d_w$, then the isomorphism is given by
\[
\begin{tikzcd}[row sep=-9pt,/tikz/column 1/.append style={anchor=base east},/tikz/column 2/.append style={anchor=base west}]
{\CC[x,y,z]_{w,d}}\rar&H^0(\PP^2_w, \cO_{\PP^2_w}(D))\\
F\rar[mapsto]&\dfrac{F}{H}.
\end{tikzcd}
\]
In this situation the Euler characteristic $\chi(\PP^2_w, \cO_{\PP^2_w}(D))$
is concentrated in degree zero and by the Riemann-Roch formula one has
\[
\chi(\PP^2_w, \cO_{\PP^2_w}(D)) = 1 + \frac{1}{2} D \cdot (D - K_{\PP^2_w}) + R_{\PP^2_w}(D),
\]
where $K_{\PP^2_w}$ is the canonical divisor of $\PP^2_w$ and $R_{\PP^2_w}(D)$ is the so-called correction term.

Going back to our example, Bézout's theorem tells us that the first term in the previous formula for the
Euler characteristic can easily be computed as
\[
\frac{1}{2} D \cdot (D - K_{\PP^2_w}) = \frac{d_w (d_w+|w|)}{2\bar{w}}
= \frac{1528 \cdot (1528 + 19 + 77 + 12)}{2 \cdot 19 \cdot 77 \cdot 12}
= \frac{312476}{4389} \approx 71.195.
\]
In this case the addend $1+\frac{d (d+|w|)}{2\bar{w}} = \frac{316865}{4389} \approx 72.195$ already provides
a good approximation of the problem. However, in general if the numbers $w_0,w_1,w_2,d$ are big enough, then this
quadratic term may not suffice and the correction term really matters.

\subsection{The correction term \texorpdfstring{$R_{\PP^2_w}(D)$}{Rpw2D}}\label{sec:correction-term}

Since $\PP^2_w$ has three singular points at most, for weighted projective planes this term has simply the shape
\[
R_{\PP^2_w}(D) = R_{X(p;q,r)}(D) + R_{X(q;p,r)}(D) + R_{X(r;p,q)}(D).
\]
In our example one gets
\[
\begin{aligned}
R_{\PP^2_w}(D) &= R_{X(19; 77,12)}(1528) + R_{X(77; 19,12)}(1528) + R_{X(12; 19,77)}(1528) \\
&=  - \frac{7}{19} - \frac{38}{77} - \frac{4}{3}
= - \frac{9635}{4389}
\approx - 2.195.
\end{aligned}
\]
Therefore the number of solution of our initial problem is precisely
\[
1 + \frac{1}{2} D \cdot (D - K_{\PP^2_w}) + R_{\PP^2_w}(D) = 1 + \frac{312476}{4389} - \frac{9635}{4389} = 70,
\]
as it was computed by hand at the beginning.

We discuss the calculation of one of the summands in $R_{\PP^2_w}(D)$, namely $R_{X(19; 77,12)}(1528)$.
Recall the main properties of the local correction term for cyclic quotient singularities that allow for the calculation:
\begin{enumerate}[label=\rm(R\arabic{enumi})]
\item\label{R1} $R_{X(d;a,b)}(k) = R_{X(d;a,b)}(k \mod d)$, \vspace{2.5pt}
\item\label{R2} $R_{X(d;a,b)}(k) = R_{X(d;1,a^{-1}b \! \mod d)}(a^{-1} k)$,
\item\label{R3} $R_{X(d;1,q)}(k) = -R_{X(q;1,d \! \mod q)}(k \mod q) - \fpart{\frac{k}{q}} - \frac{k(k+1+q-d)}{2dq}$,
\end{enumerate}
where $\fpart{\frac{k}{q}}$ denotes the fractional part of the quotient. First, one applies \ref{R1} and \ref{R2} to pass from a general
cyclic quotient space to one of the form $X(d;1,q)$ so that one can use~\ref{R3},
\[
R_{X(19; 77,12)}(1528) \stackrel{(1)}{=} R_{X(19; 77,12)}(8) \stackrel{(2)}{=} R_{X(19; 1,12)}(8).
\]
Since the Euclidean algorithm for $(19,12)$ involves five divisions, the number of operations for computing
$R_{X(19;1,12)}(k)$ will be the same as those for computing $R_{X(13;1,8)}(k)$ corresponding to the Fibonacci numbers
$F_7 = 13$ and $F_6 = 8$, which is the worst case of the algorithm from complexity point of view.
By~\ref{R3}
\[
\begin{aligned}
R_{X(19;1,12)}(8) &= - R_{X(12;1,7)}(8) - \fpart{\frac{8}{12}} - \frac{8 \cdot (8+1+12-19)}{2 \cdot 19 \cdot 12}
= - R_{X(12;1,7)}(8) - \frac{40}{57}, \\[5pt]
R_{X(12;1,7)}(8) &= - R_{X(7;1,5)}(1) - \fpart{\frac{8}{7}} - \frac{8 \cdot (8+1+7-12)}{2 \cdot 12 \cdot 7}
= - R_{X(7;1,5)}(1) - \frac{1}{3}, \\[5pt]
R_{X(7;1,5)}(1) &= - R_{X(5;1,2)}(1) - \fpart{\frac{1}{5}} - \frac{1 \cdot (1+1+5-7)}{2 \cdot 7 \cdot 5}
= - R_{X(5;1,2)}(1) - \frac{1}{5}, \\[5pt]
R_{X(5;1,2)}(1) &= - R_{X(2;1,1)}(1) - \fpart{\frac{1}{2}} - \frac{1 \cdot (1+1+2-5)}{2 \cdot 5 \cdot 2}
= - R_{X(2;1,1)}(1) - \frac{9}{20}, \\[5pt]
R_{X(2;1,1)}(1) &= - \cancel{R_{X(1;1,0)}(0)} - \fpart{\frac{1}{1}} - \frac{1 \cdot (1+1+1-2)}{2 \cdot 2 \cdot 1} = -\frac{1}{4}.
\end{aligned}
\]
The procedure finishes here since in general $R_{X(1;a,b)}(k)$ is always zero. Summarizing 
\[
R_{X(19;1,12)}(8) = - \frac{40}{57} + \frac{1}{3} - \frac{1}{5} + \frac{9}{20} - \frac{1}{4} = - \frac{7}{19}. 
\]
The sum of the last three fractions corresponds to $R_{X(7;1,5)}(1)$ and it is already zero.
One could have stopped the method at that step. This is a consequence of a general property of
the correction term, namely $R_{X(r;s,t)}(r-s-t) = 0$.

\subsection{The \texorpdfstring{$\Delta_X$}{DX}-invariant}

There is another way to interpret the correction term for cyclic quotient singularities by using
the $\Delta_X$-invariant as follows
\[
R_{X(d;a,b)}(k) = - \Delta_{X(d;a,b)}(-k),
\]
where by definition
\[
\Delta_{X(p;-1,q)}(k) = A^{(p,q)}_r - \delta^{(p,q)}_r
\quad
(r = q^{-1}k \mod p)
\]
and
\[
A^{(p,q)}_r = \# \{ (i,j) \in \ZZ^2_{\geq 1} \mid pi+qj \leq qr \},
\qquad
\delta^{(p,q)}_r = \frac{r(qr-p-q+1)}{2p}.
\]
Recall that $A^{(p,q)}_r$ is the number of integral points in a triangle while $\delta^{(p,q)}_r$ is the expected
number of integral points if the triangle had integral vertices.

\begin{figure}[ht]
\centering
\begin{tikzpicture}[xscale=1, yscale=1]
\begin{axis}[domain=-1:5, xmin=-1, xmax=5.1, ymin=-2, ymax=13.0, xlabel={}, ylabel={},
axis x line=center, axis y line=center, xtick={1,2,3,4}, ytick={1,3,5,7,9,12}, samples=100]
\draw[white, fill=blue, opacity=0.25] (1,1) -- (1,65/7) -- (77/19,1) -- cycle;
\draw[blue, ultra thick] (1,1) -- (1,65/7) -- (77/19,1) -- cycle;
\draw (4.4,-1) node {$\frac{84}{19}$};
\draw (3.5,7) node {$19x+7y=84$};
\draw[fill=black] (1,65/7) node [above right=-1] {$(1,\frac{65}{7})$} circle (3pt);
\draw[fill=black] (77/19,1) node [above right=-3] {$(\frac{77}{19},1)$} circle (3pt);
\addplot[blue] (x, -19/7*x + 12);
\draw [draw=black, fill=green] (1,1) circle (2.5pt);
\draw [draw=black, fill=green] (1,2) circle (2.5pt);
\draw [draw=black, fill=green] (1,3) circle (2.5pt);
\draw [draw=black, fill=green] (1,4) circle (2.5pt);
\draw [draw=black, fill=green] (1,5) circle (2.5pt);
\draw [draw=black, fill=green] (1,6) circle (2.5pt);
\draw [draw=black, fill=green] (1,7) circle (2.5pt);
\draw [draw=black, fill=green] (1,8) circle (2.5pt);
\draw [draw=black, fill=green] (1,9) circle (2.5pt);
\draw [draw=black, fill=green] (2,1) circle (2.5pt);
\draw [draw=black, fill=red] (2,2) circle (2.5pt);
\draw [draw=black, fill=red] (2,3) circle (2.5pt);
\draw [draw=black, fill=red] (2,4) circle (2.5pt);
\draw [draw=black, fill=red] (2,5) circle (2.5pt);
\draw [draw=black, fill=red] (2,6) circle (2.5pt);
\draw [draw=black, fill=green] (3,1) circle (2.5pt);
\draw [draw=black, fill=red] (3,2) circle (2.5pt);
\draw [draw=black, fill=red] (3,3) circle (2.5pt);
\draw [draw=black, fill=green] (4,1) circle (2.5pt);
\end{axis}
\end{tikzpicture}
\caption{Points associated with $A^{(19,7)}_{12}$, i.e.~$19i+7j \leq 84$, $i,j \geq 1$.}
\label{fig:A19-7-12}
\end{figure}
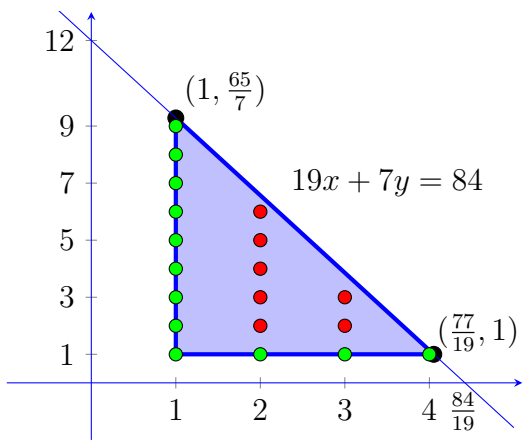

In our example $p=19$, $q=7$, $k=8$, $r = (q^{-1}k \mod p) = 12$ and one obtains
\[
R_{X(19;1,12)}(8) = -\Delta_{X(19;-1,7)}(8) = - A^{(19,7)}_{12} + \delta^{(19,7)}_{12} = -19 + \frac{354}{19} = -\frac{7}{19},
\]
as it was claimed above. Figure \ref{fig:A19-7-12} shows the triangle associated with $A^{(19,7)}_{12}$
and the $19$ integral points, $7$ interior and $12$ on the boundary.
According to Pick's theorem, if the triangle had had integral vertices, then its area would have been
$A = i + \frac{b}{2} - 1 = 12$. However, the true area is $A \approx 12.65$.

\subsection{Minimal generic curves on quotient singularities}\label{sec:generic-curves}

For cyclic quotient singularities the $\Delta_X$-invariant has an intrinsic geometric meaning.
After choosing $f \in \cO_X(k)$ reduced, the rational number $\Delta_X(k)$ can be described by the difference
\begin{equation}\label{eq:DeltaXk}
\Delta_X (k) = \delta_X^{\text{top}}(f) - \kappa_X(f)
\end{equation}
where $\delta_X^{\text{top}}$ is the topological delta invariant and $\kappa_X$
is the analytic kappa invariant of the singularity.
Note that the choice of a reduced $f \in \cO_X(k)$ does not affect the result of $\Delta_X(k)$. Roughly speaking $\delta_X^{\text{top}}(f)$
is a combinatorial formula once an embedded $\QQ$-resolution for $f=0$ is known while $\kappa_X(f)$ is related to counting points in a lattice.
Since the latter is a hard problem, in order to compute the $\Delta_X$-invariant, the idea is to choose $f \in \cO_X(k)$ so that
$\kappa_X(f)$ becomes very easy to calculate and then compute $\delta_X^{\text{top}}(f)$ for such a particular~$f$.
This motives the notion of minimal generic curves in $X(d;a,b)$, see \cite{CMO16}. This notion was generalized for
other type of singularities in~\cite{CLMN22b}, see also~\cite{CLMN22b}.

In our example $X = X(19;1,12)$ and $k = (-8 \mod 19) = 11$. In order to get the numerical data associated with the Hirzebruch-Jung
continued fraction of $\frac{19}{12}$, we apply the Euclidean algorithm (by excess)
\[
\begin{aligned}
19 &= 2 \cdot 12 - 5, \\
12 &= 3 \cdot 5 - 3, \\
5 &= 2 \cdot 3 - 1, \\
3 &= 3 \cdot 1 - 0.
\end{aligned}
\]
This gives $\mathbf{q} = [12,5,3,1]$, $\bar{\mathbf{q}} = [1,2,5,8]$, $\mathbf{c} = [2,3,2,3]$, $[k] = [0,2,0,1]$, $n=4$.
One chooses as minimal generic curve in $\cO_{X(19;1,12)}(11)$ the following
\[
f^{\text{gen}} = (x^{q_2} + y^{\bar{q}_2}) (x^{q_2} - y^{\bar{q}_2}) (x^{q_4} + y^{\bar{q}_4})
= (x^5 - y^2) (x^5 - 2 y^2) (x - y^8).
\]
The kappa invariant is the number of branches of $f^{\text{gen}}$ minus 1, that is, $\sum_{i=1}^4 k_i - 1 = 2$.
Note that the number of branches of $f^{\text{gen}}$ in $\cO_{\CC^2}$ is also $r=3$. Then one can use the formula
\[
\delta_{X(19;1,12)}^{\text{top}} (f^{\text{gen}}) = \frac{\delta^{\text{top}}_{\CC^2}(f^{\text{gen}})}{d} + r \frac{d-1}{2d}
= \frac{18}{19} + 3 \frac{18}{2\cdot 19} = \frac{45}{19}
\]
and finally according to~\eqref{eq:DeltaXk}
\[
R_{X(19;1,12)}(8) = -\Delta_{X(19;1,12)}(11)
= - \frac{45}{19} + 2 = - \frac{7}{19}.
\]
For computing $\delta^{\text{top}}_{\CC^2}(f^{\text{gen}})=18$ above one uses the standard techniques from plane curves,
such as resolution of singularities or the property $\delta^{\text{top}}_{\CC^2}(fg) = \delta^{\text{top}}_{\CC^2}(f) + \delta(g)
+ i(f,g)$, where $i(f,g)$ is the local intersection multiplicity, cf.\S\ref{sec:resolution-sing} below.

Unfortunately this method for computing the $\Delta_X$-invariant may be very inefficient, in comparison to the one described
in \S\ref{sec:correction-term}, if the length of the Euclidean algorithm by excess is very long. This happens for instance
when the quotient space is of the form $X(d;1,d-1)$.

\begin{remark}
Nowadays it is known that the analytic kappa invariant coincides with the $\delta$-invariant of a given curve
$\cC = \{ f = 0 \} \subset X = X(d;1,q)$. Therefore
\[
\delta(\cC) = \kappa_X(f) = \delta_X^{\text{top}}(f) - \Delta_X (k)
= \delta_X^{\text{top}}(f) - \delta_X^{\text{top}}(f^{\text{gen}}) + \kappa_X(f^{\text{gen}}),
\]
where $f \in \cO_X(k)$ is reduced and $f^{\text{gen}}$ is a minimal generic curve in $\cO_X(k)$.
This way the computation of such an important invariant of a curve relies on the calculation of
the minimal generic curves.
\end{remark}

\subsection{Resolution of singularities}\label{sec:resolution-sing}

Consider $\pi: (\tilde{X},E) \to (X,x)$ a resolution of $(X,D)$ at $x$,
where $E$ is the exceptional part of the resolution, $\pi^* D = \hat{D} + E_D$,
$\hat{D}$ is the strict transform of $D$, and $E_D$ is its exceptional part.
Denote by $K_{\tilde{X}}$, $K_X$, and $K_\pi$ the canonical divisor of $\tilde{X}$,
the canonical divisor of $X$, and the relative canonical divisor of $\pi$, respectively.
One has the relation $K_{\tilde{X}} = \pi^* K_X + K_\pi$.
Then
\[
\delta^{\text{top}}_X(D) := \frac{1}{2} E_D \cdot (\hat{D}+K_{\tilde{X}})
= \frac{1}{2} \hat{D} \cdot (E_D-K_\pi)
= - \frac{1}{2} E_D \cdot (E_D-K_\pi).
\]
In our example, if $D = \{ f^{\text{gen}} = 0 \} \subset X(19;1,12)$, then
the Hirzebruch-Jung resolution consists of $4$ exceptional divisors forming
a bamboo-shaped graph, see Figure~\ref{fig:resolution-19-1-12}.

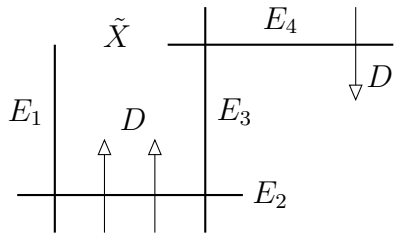
\begin{figure}[ht]
\centering
\begin{tikzpicture}[scale=1]
\draw[thick] (0.5,0) -- node[above=10, left]{$E_1$} (0.5,2.5);
\draw[-open triangle 45] (1.16,0) -- node[left]{} (1.16,1.25);
\draw[-open triangle 45] (1.83,0) -- node[above=25, left=-1]{$D$} (1.83,1.25);
\draw[thick] (0,0.5) -- node[right=42]{$E_2$} (3,0.5);
\draw[thick] (2.5,0) -- node[above=3, right]{$E_3$} (2.5,3);
\draw[thick] (2,2.5) -- node[above]{$E_4$} (5,2.5);
\draw[-open triangle 45] (4.5,3) -- node[below right]{$D$} (4.5,1.75);
\node[below right] at (1,3) {$\tilde{X}$};
\end{tikzpicture}
\caption{Resolution of $X=X(19;1,12)$ and $f^{\text{gen}} \in \cO_X(11)$.}
\label{fig:resolution-19-1-12}
\end{figure}

Denote by $\gamma_1 = x^{12} + y$,
$\gamma_2 = x^5 + y^2$, $\gamma_3 = x^3 + y^5$, $\gamma_4 = x + y^8$ the four
curvettes. Then the exceptional part is
\[
E_D = \sum_{i=1}^4 (D \cdot \gamma_i) E_i
= \frac{11}{19} E_1 + \frac{22}{19} E_2 + \frac{17}{19} E_3 + \frac{12}{19} E_4
\]
and the relative canonical divisor is
\begin{equation}\label{eq:Kpi-quotient}
K_\pi = \sum_{i=1}^n \left( \frac{q_i + \bar{q}_i}{d} - 1 \right) E_i
= - \frac{6}{19} E_1 - \frac{12}{19} E_2 - \frac{11}{19} E_3 - \frac{10}{19} E_4.
\end{equation}

The self-intersection number of $E_i$ are given in the vector $\mathbf{c} = [2,3,2,3]$, \S\ref{sec:generic-curves}.
So the intersection matrix of the resolution is 
\[
M =
\begin{pmatrix}
-2 & 1 & 0 & 0 \\
1 & -3 & 1 & 0 \\
0 & 1 & -2 & 1 \\
0 & 0 & 1 & -3
\end{pmatrix}
\]
and hence
\[
\delta_{X(19;1,12)}^{\text{top}}(f^{\text{gen}})
= - \frac{1}{2} \begin{pmatrix} \frac{11}{19} & \frac{22}{19} & \frac{17}{19} & \frac{12}{19} \end{pmatrix} M
\begin{pmatrix} \frac{17}{19} \\[7pt] \frac{34}{19} \\[7pt] \frac{28}{19} \\[7pt] \frac{22}{19} \end{pmatrix}
= \frac{45}{19}.
\]

\begin{remark}
In this particular example a resolution for the minimal generic curve can also be computed simply in
one step and this provides another way to calculate its topological delta invariant.
Yet another way is to use $\delta_X^{\text{top}}(fg) = \delta_X^{\text{top}}(f) + \delta_X^{\text{top}}(g) + i_X(f,g)$.
The details are left to the reader.
\end{remark}

\subsection{The decompositions of \texorpdfstring{$\cO_X(k)$}{OXk} and the null-submodule}

Let $X(d;1,q)$ and assume $f \in \cO_X(k)$. The analytic kappa invariant is defined as
\[
\kappa_X(f) = \dim_\CC \frac{\cO_X(k-1-q)}{M_{f,\pi}^{\text{nul}}},
\]
where $M^{\text{nul}}_{f,\pi}$ is the null-submodule
\[
M^{\text{nul}}_{f,\pi} = \{ h \in \cO_{X}(k-1-q) \mid \pi^{*} ((h) - (f)) + K_\pi \geq 0 \}.
\]
The kappa invariant does not depend on the chosen resolution.

In our example $f^{\text{gen}} \in \cO_{X(19;1,12)}(11)$ and it can be resolved using a weighted blow-up
with weights $(2,5)$, cf.~\S\ref{sec:resolution-sing}, see Figure~\ref{fig:19-1-12_one-step}.

\begin{figure}[ht]
\centering
\begin{tikzpicture}[scale=1.0]
\definecolor{bluegreen2}{RGB}{0, 85, 127}
\colorlet{col1}{blue!40}
\colorlet{col2}{red!40}
\colorlet{col3}{violet!40}
\colorlet{colsing}{bluegreen2}
\colorlet{colsingtext}{colsing}
\draw[very thick] (-3.5,0)--(3.5,0) node [right, xshift=3pt] {$E$};
\node[color=colsing] (O1) at (-2.5,-0.02) {\Large $\bullet$};
\node[below, yshift=-3pt, text=colsingtext] at (O1) {$(2)$};
\node (X1) at (0.5,0) {$\times$};
\draw[-open triangle 45] (0.5, -0.75) -- (0.5,1);
\node[color=colsing] (O2) at (2.5,-0.02) {\Large $\bullet$};
\draw[-open triangle 45] (2.5, -0.75) -- (2.5,1);
\node (X2) at (-0.5,0) {$\times$};
\draw[-open triangle 45] (-0.5, -0.75) -- (-0.5,1);
\node[below, xshift=10pt, yshift=-3pt, text=colsingtext] at (O2) {$(5)$};
\end{tikzpicture}
\caption{Resolution of $X=X(19;1,12)$ and $f^{\text{gen}} \in \cO_X(11)$ in one step.}
\label{fig:19-1-12_one-step}
\end{figure}
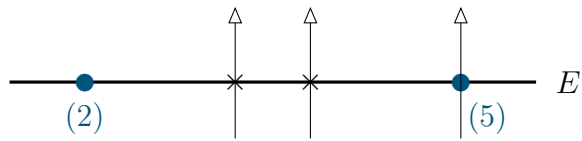

The exceptional divisor has multiplicity $\frac{22}{19}$ and the relative canonical
divisor has multiplicity $-\frac{12}{19}$. Then the condition for $h \in \cO_X(17)$
to be in $M^{\text{nul}}_{f,\pi}$ is
\[
\mult_{E} ( \pi^* h ) \geq \frac{34}{19}.
\]
Using the charts of $\pi$, one writes the multiplicity of the exceptional divisor in $\pi^* h$ as
$\frac{1}{19} \ord h(x^2,x^5,y)$. Then
\[
M^{\text{nul}}_{f^{\text{gen}},\pi} = \left\{ \sum_{i,j} a_{ij} x^i y^j \ \Big| \ i+12j \equiv 17 \!\!\! \mod 19, \ 2i+5j \geq 34 \right\}.
\]
Finally,
\[
\frac{\cO_X(k-1-q)}{M_{f,\pi}^{\text{nul}}} \simeq \CC \langle x^5 y, y^3 \rangle
\]
and the kappa invariant is $2$ that coincides with the number of local branches of $f^{\text{gen}}$ minus one,
as it was claimed in \S\ref{sec:generic-curves}. In general the null-submodule is not monomial if the curve
is not minimal generic. According to the theory $\cO_X(17) = \cO_X(12) \otimes  \cO_X(5)$ and
$M^{\text{nul}}_{f,\pi} = \cO_X(5)^{\otimes 3} \otimes \cO_X(1)^{\otimes 2}$ and the quotient has dimension $2$
as $\CC$-vector space.

\subsection{On Blache's 1st question}

Let $X = X(d;1,q)$ with $\gcd(d,q)=1$ and denote by $K_X$ its canonical divisor.
Blache observed \cite[\S8.5]{Bla95} that the behavior of $R_X(\ell K_X)$ had to do with a parabola.
More precisely, denote by $I = \min \{ m \in \NN \mid m K_X \text{ is Cartier} \}$;
it can arithmetically be expressed as $I = \frac{d}{\gcd(d,q+1)}$.
Let $f_{I}: \RR \to \RR$ be the polynomial of degree two
such that $f_{I}(1) = f_{I}(I) = 0$ and $f'_{I}(1) = 1$. Hence
$f_I(x) = \frac{(x-1)(I-x)}{I-1}$. He asked whether $|R_X(\ell K_x)| < f_{I}(\ell)$
for $\ell = 2, \ldots, I-1$. We showed that
\[
| R_{X}(\ell K_X) | \leq \frac{(\ell-1)(I-\ell)}{I}
\]
which answers the question in the positive.
In Figure~\ref{fig:Blache-1}, it is shown the shape of $R_X(x K_X)$ (in red) for
different quotient spaces in comparison to the parabola $\frac{(x-1)(I-x)}{I}$ (in blue).

\begin{figure}[ht]
\centering
\begin{tikzpicture}[xscale=0.65, yscale=0.60]
\begin{axis}[domain=1:19, xmin=0, xmax=21.0, ymin=0, ymax=5.0, xlabel={}, ylabel={},
axis x line=center, axis y line=center, xtick={1,5,10,15,19}, ytick={1,2,3,4}, samples=100]
\addplot[blue, ultra thick] (x, -1/19*x^2 + 20/19*x - 1); 
\addplot coordinates {(1, 0) (2, 3/19) (3, 9/19) (4, 1/19) (5, 11/19) (6, 7/19) (7, 6/19) (8, 8/19) (9, 13/19)
(10, 2/19) (11, 13/19) (12, 8/19) (13, 6/19) (14, 7/19) (15, 11/19) (16, 1/19) (17, 9/19) (18, 3/19) (19, 0)};
\addplot[green] (x, -9/608*x^2 + 45/152*x - 9/32);
\draw (10,4.75) node {$X(19;1,12)$};
\draw [draw=black, fill=white] (14,7/19) circle (4.0pt);
\end{axis}
\end{tikzpicture} \quad
\begin{tikzpicture}[xscale=0.65, yscale=0.60]
\begin{axis}[domain=1:19, xmin=0, xmax=21.0, ymin=0, ymax=5.0, xlabel={}, ylabel={},
axis x line=center, axis y line=center, xtick={1,5,10,15,19}, ytick={1,2,3,4}, samples=100]
\addplot[blue, ultra thick] (x, -1/19*x^2 + 20/19*x - 1); 
\addplot coordinates {(1, 0) (2, 15/19) (3, 26/19) (4, 33/19) (5, 36/19) (6, 35/19) (7, 30/19) (8, 21/19) (9, 8/19)
(10, 9/19) (11, 8/19) (12, 21/19) (13, 30/19) (14, 35/19) (15, 36/19) (16, 33/19) (17, 26/19) (18, 15/19) (19, 0)}; 
\draw (10,4.75) node {$X(19;1,1)$};
\end{axis}
\end{tikzpicture} \quad
\begin{tikzpicture}[xscale=0.65, yscale=0.60]
\begin{axis}[domain=1:19, xmin=0, xmax=21.0, ymin=0, ymax=5.0, xlabel={}, ylabel={},
axis x line=center, axis y line=center, xtick={1,5,10,15,19}, ytick={1,2,3,4}, samples=100]
\addplot[blue, ultra thick] (x, -1/10*x^2 + 11/10*x - 1); 
\addplot coordinates {(1, 0) (2, 4/5) (3, 7/5) (4, 9/5) (5, 2) (6, 2) (7, 9/5) (8, 7/5) (9, 4/5) (10, 0)};
\draw (10,4.75) node {$X(20;1,1)$};
\end{axis}
\end{tikzpicture}
\caption{Comparison between $R_X(xK_X)$ and $\frac{(x-1)(I-x)}{I}$.}
\label{fig:Blache-1}
\end{figure}
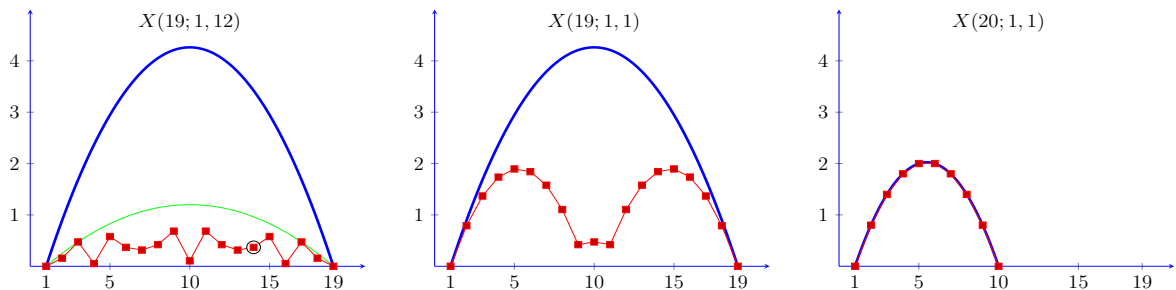

In the previous sections we have seen that $R_{X(19;1,12)}(8) = - \frac{7}{19}$.
Note that $k = 8$ corresponds to $\ell = (8 \cdot 13^{-1} \mod 19) = 14$ that has been marked
in the drawing with a circle. Although the bound for $R_X(\ell K_X)$ does not seem to be very
accurate in general, it is sharp for spaces of the form $X(d;1,1)$ when $d$ is even.
It remains open to find a parabola that better bounds the correction term.
For instance, the best parabola for $X(19;1,12)$ is $f(x) = \frac{9(x-1)(19-x)}{608}$ (in green)
that is characterized by $f(1)=0$, $f(I)=0$, $f'(1) = \frac{81}{304} \approx 0.27$
whose slope at $x = 1$ is much smaller than the one suggested by Blache.

\subsection{On Blache's 2nd question}

In order to understand the behavior of $R_X(\ell K_X)$ he also asked whether $|R_X((\ell+1)K_X) - R_X(\ell K_X)| < 1$
for all $\ell \in \ZZ$. We noticed that this question was related to the log-canonical threshold of $X$ with respect
to the maximal ideal. In particular, we showed that
\begin{equation}\label{eq:RminusR}
|R_X((\ell+1)K_X) - R_X(\ell K_X)| \leq 1 - \lct(X,\mathfrak{m})
\end{equation}
which answers the question in the positive for quotient singularities since $\lct(X,\mathfrak{m}) = \min\{ \frac{q_i+\bar{q}_i}{d}
\mid i = 1,\ldots,n \} \geq \frac{2}{d}$, cf.~\eqref{eq:Kpi-quotient}.

In our example $X(19;1,12)$ the list of differences $|R_X((\ell+1)K_X) - R_X(\ell K_X)|$ for $\ell = 1,\ldots,d-1$ are
\[
\frac{3}{19}, \frac{6}{19}, \frac{10}{19}, \mathbf{\frac{12}{19}}, \frac{4}{19}, \frac{1}{19}, \frac{2}{19}, \frac{5}{19}, \frac{11}{19},
\frac{11}{19}, \frac{5}{19}, \frac{2}{19}, \frac{1}{19}, \frac{4}{19}, \mathbf{\frac{12}{19}}, \frac{10}{19}, \frac{6}{19}, \frac{3}{19}.
\]
Note that the log-canonical threshold is $\lct(X,\mathfrak{m}) = \frac{q_2+\bar{q}_2}{d} = \frac{7}{19}$ and then the bound
provided in~\eqref{eq:RminusR} is sharp and it is reached by $\ell = 4$ and $\ell = 15$. This integers are the solution
of $\pm \ell (1+q) = q_2 \mod d$.


\begin{appendices}
\section{Proof of a simple formula}\label{sec:formulaq}

Let $\zeta_q$ be a $q$th root of unity. During the paper we have used several times that
\[
\sum_{i=1}^{q-1} \frac{1}{1-\zeta_q^i} = \frac{q-1}{2}.
\]
Here we provides a simple proof of this fact.
Note that $t^q - 1$ can be decomposed into irreducible factors as follows
\[
t^q - 1 = \prod_{i=0}^{q-1} (t-\zeta^i).
\]
The derivative with respect to $t$ gives
\[
qt^{q-1} = \sum_{i=0}^{q-1} \frac{t^q-1}{t-\zeta^i} \quad \Longrightarrow \quad
\frac{qt^{q-1}}{t^q-1} = \sum_{i=0}^{q-1} \frac{1}{t-\zeta^i}.
\]
Therefore
\[
\sum_{i=1}^{q-1} \frac{1}{1-\zeta_q^i} =
\lim_{t \to 1} \sum_{i=1}^{q-1} \frac{1}{t-\zeta_q^i} =
\lim_{t \to 1} \left( \frac{qt^{q-1}}{t^q-1} - \frac{1}{t-1} \right) =
\frac{q-1}{2}.
\]
The last limit is a simple calculus exercise.
\end{appendices}


\end{document}